\documentclass[11pt,a4paper]{article}

\usepackage[comma]{natbib}
\usepackage{anysize}
\usepackage{color}
\usepackage{amsmath,amssymb,amsthm,amsopn,bm,mathrsfs}
\usepackage{enumitem}
\usepackage{dcolumn}
\newcolumntype{=}{D{=}{=}{-1}}

\DeclareMathOperator{\E}{\boldsymbol{\mathbb{E}}}

\DeclareMathOperator{\tr}{tr}

\DeclareMathOperator{\re}{Re}

\renewcommand{\today}{\begingroup
\number \day\space  \ifcase \month \or January\or February\or
March\or April\or May\or June\or July\or August\or September\or
October\or November\or December\fi \space  \number \year \endgroup}

\newcommand{\mat}[1]{\mathbf{#1}}

\newcommand{\D}{\mathsf{D}}
\newcommand{\Hessian}{\mathsf{H}}
\newcommand{\bA}{{\mat A}}
\newcommand{\bH}{{\bf H}}

\newcommand{\bI}{{\bf I}}

\newcommand{\ba}{{\bm a}}
\newcommand{\bb}{{\bm b}}
\newcommand{\bc}{{\bm c}}
\newcommand{\bC}{{\bf C}}
\newcommand{\be}{{\bm e}}
\newcommand{\bbf}{{\bm f}}
\newcommand{\bg}{{\bm g}}

\newcommand{\bj}{{\bm j}}

\newcommand{\bbm}{{\bm m}}

\newcommand{\bt}{{\bm t}}
\newcommand{\bu}{{\bm u}}

\newcommand{\bv}{{\bm v}}

\newcommand{\bx}{{\bm x}}

\newcommand{\by}{{\bm y}}
\newcommand{\bz}{{\bm z}}
\newcommand{\bB}{{\bf B}}

\newcommand{\bX}{{\bf X}}

\newcommand{\bmu}{{\bm \mu}}

\renewcommand{\S}{\bm{\mathcal S}}
\newcommand{\K}{{\bf K}}
\newcommand{\bSigma}{{\bm \Sigma}}

\newcommand{\diff}{\mathsf{d}}
\renewcommand{\vec}{\operatorname{vec}}

\newenvironment{bbmatrix}{\left[\begin{array}{c}}{\end{array}\right]}
\theoremstyle{plain}
\newtheorem{lemma}{Lemma}
\newtheorem{theorem}{Theorem}

\marginsize{3cm}{3cm}{3cm}{3cm}

\title{{Higher order differential analysis with vectorized derivatives}}

\author{Jos\'e E. Chac\'on\footnote{Departamento de
Matem\'aticas, Universidad de Extremadura, E-06006 Badajoz, Spain. E-mail:
{\tt jechacon@unex.es}} \and Tarn Duong\footnote{F-75000 Paris, France. Email:
{\tt tarn.duong@gmail.com}}}

\begin{document}

\maketitle

\begin{abstract}
\noindent Higher order derivatives of functions of vector variables are structured high dimensional objects which lend themselves to many alternative representations, with the most popular being multi-index, matrix and tensor representations. The choice between them depends on the desired analysis since each presents its own advantages and disadvantages. In this paper, we highlight a vectorized representation, in which higher order derivatives are expressed as vectors. This structure allows us to construct an elegant and rigorous algebra of vector-valued functions of vector variables, which would be unwieldy, if not impossible, to do so using the other representations. The fundamental results that we establish for this algebra of differentials are the identification theorems, with concise existence and uniqueness properties, between differentials and derivatives of an arbitrary order. From these fundamental identifications, we develop further analytic tools, including a Leibniz rule for the product of functions, and a chain rule (Fa\`a di Bruno's formula) for the composition of functions. {We then exhibit novel expressions for the higher order derivatives of some important functions, namely the vector/matrix monomial, matrix trace, matrix inverse and matrix determinant.} To complete our exposition, we illustrate how existing results (such as Taylor's theorem) can be incorporated into and generalized within our framework of higher order differential calculus.
\end{abstract}



\newpage

\section{Introduction}

Classical references on matrix differential calculus, including \citet{Sea82}, \citet{Bas83} and \citet{Gra83}, represent higher order derivatives with the multi-index notation. Whilst multi-indices allow for the concise statement of these derivatives, they lack important algebraic properties that facilitate the generalization of certain well-known fundamental results in real differential calculus. {Vector- and matrix-based derivatives possess some of these desired algebraic properties, and they are thoroughly exposited in the seminal monographs of \citet{Sch17} and \citet{MN19}, and exploited to illustrate its applications regarding linear models, maximum likelihood and principal component analysis. This firmly establishes matrix analysis in the set of analytic tools for advanced statistical methodologies beyond the archetypal cases of ordinary least squares or simple linear regression. Other important contributions are included in \citet{W02}, where the usefulness of vector differential calculus is illustrated for generalized linear models, kriging, and several types of regression models.} 

The aforementioned monographs focus on the first and second-order derivatives, mainly because these are the most important to tackle initially, but also due to a lack of a unifying framework to treat higher order derivatives. However, two recent books by \citet{T13} and \citet{T21} highlight the necessity of developing such a framework for higher order analysis. {In this paper, we focus on an alternative representation as vectorized derivatives, introduced by \citet{Hol96a}, which do possess the full range of desired algebraic properties. Consequently, we are able to fill in many of the gaps in the currently sparse suite of analytical tools for higher order derivatives of vector-valued functions of vector variables. Thus we assert that vectorized dervatives are a viable alternative to the more well-established formulations.}

We begin our motivation by observing that the derivatives of a scalar-valued function of a vector variable $f\colon\mathbb R^d\to\mathbb R$ are intimately linked to the task of approximating such a function in a neighborhood of a given point ${\bc}\in\mathbb R^d$ \citep{Mag10}.  The almost universally adopted definition of differentiability is as follows. The function $f$ is said to be differentiable at ${\bc}$ if there exists a linear function $\mathsf{d}f(\bc;\cdot)\colon\mathbb R^d\to\mathbb R$ such that, for all $\bm u\in\mathbb R^d$,
\begin{equation}\label{eq:dif1}
f({\bc}+{\bm u})=f(\bc)+\mathsf{d}f(\bc;\bm u)+\re_{\bc}(\bm u)
\end{equation}
with the remainder satisfying $\re_{\bc}(\bm u)/\|\bm u\|\to0$ as $\bm u\to{\bm 0}$, where $\|\bm u\|=(\bm u^\top\bm u)^{1/2}$ denotes the norm of $\bm u$. See, for example, \citet[][p.~67]{E73}. In this case, the function $\mathsf{d}f(\bc;\cdot)$ is known as the differential of $f$ at $\bc$, and the so-called first identification theorem \citep[][Section~5.8]{MN19} shows that $\mathsf{d}f(\bc;\bm u)=\D f(\bc)^\top\bm u$, where the vector $\D f(\bc)\in\mathbb R^d$ is the gradient (or first derivative) of $f$ at $\bc$. This vector is unique, in the sense that if there is another $\bm a\in\mathbb R^d$ such that $\mathsf{d}f(\bc;\bm u)={\bm a}^\top\bm u$ for all $\bm u\in\mathbb R^d$, then necessarily $\bm a=\D f(\bc)$. Moreover, omitting the remainder terms of Equation~\eqref{eq:dif1}, $f({\bc}+{\bm u})\approx f(\bc)+\D f(\bc)^\top\bm u$ is the best linear (or first-order) approximation of $f$ in a neighborhood of $\bc$.

Analogously, if $f$ is differentiable at every point in some neighborhood of $\bc$, and all its partial derivatives are also differentiable at $\bc$, then $f$ is said to be twice differentiable at $\bc$. In this case, then there exists a quadratic form $\diff^2f(\bc;\cdot)\colon\mathbb R^d\to\mathbb R$, called the second differential of $f$ at ${\bc}$, such that
\begin{equation}\label{eq:dif2}
f({\bc}+{\bm u})=f(\bc)+\mathsf{d}f(\bc;\bm u)+\tfrac12\diff^2f(\bc;\bm u)+\re_{\bc}(\bm u)
\end{equation}
with $\re_{\bc}(\bm u)/\|\bm u\|^2\to0$ as $\bm u\to{\bm 0}$. The second identification theorem \citep[][Section 6.8]{MN19} states that $\mathsf{d}^2f(\bc;\bm u)={\bm u}^\top\Hessian f(\bc)\bm u$, where the symmetric $d\times d$ matrix $\Hessian f(\bc)$
is called the Hessian matrix (or second derivative) of $f$ at $\bc$. In this case, uniqueness is guaranteed up to symmetrization, in the sense that if $\bB$ is another matrix such that $\mathsf{d}^2f(\bc;\bm u)={\bm u}^\top\bB\bm u$ for all $\bm u\in\mathbb R^d$, then necessarily $\Hessian f(\bc)=(\bB+\bB^\top)/2$. Omitting the remainder terms of Equation~\eqref{eq:dif2},
$f({\bc}+{\bm u})\approx f(\bc)+\D f(\bc)^\top\bm u+\tfrac12 \bu^\top\Hessian f(\bc)\bu$ is the best quadratic (or second-order) approximation of $f$ in a neighborhood of $\bc$.

First- and second-order Taylor approximations, along with their identification theorems, are well known. Indeed, for most applications a second-order approximation suffices, and so, formal generalizations of Equations~\eqref{eq:dif1} and \eqref{eq:dif2} to higher order degrees are often sparsely treated or even omitted in standard textbooks.
Nonetheless, higher order Taylor approximations are well-established research tools, and it would be erroneous to suggest that higher order derivatives and higher order approximations of functions are mere intellectual curiosities, as there are numerous applicative contexts where they are required.
For instance, third- and fourth-order approximations are involved in the analysis of local quadratic regression \citep[][Section~3]{RW94}, fourth- and sixth-order derivatives appear in the expansion of the bias and variance of density curvature matrix estimators \citep{CD10}, and more generally, expansions of distribution and density functions (e.g. Edgeworth expansions) are defined for an arbitrary order depending on the regularity of the distributional moments \citep[][Section~3.2]{KvR05}. In this sense, the recent publication of \citet{JTT21} on multivariate skewness and kurtosis highlights the current and ongoing interest in employing vectorized derivatives for higher order analysis. Furthermore, up to eighth-order derivatives are required in the analysis of bandwidth matrices for nonparametric kernel smoothers \citep{CD18}.

A widespread and elegant way to express these higher order Taylor approximations is via multi-index notation. Let $\bbm=(m_1,\dots,m_d)$ be a multi-index, that is, a vector of non-negative integers. Denote its modulus as $|\bbm|=\sum_{i=1}^d m_i$, its generalized factorial as $\bbm!=\prod_{i=1}^d m_i!$, the element-wise exponentiation of $\bx=(x_1,\dots,x_d)$ as $\bx^{\bbm}=\prod_{i=1}^dx_i^{m_i}$, and its induced partial derivative of $f$ at $\bc$ as $\partial^{\bbm}f(\bc)= (\partial^{|\bbm|}/\partial x_1^{m_1}\cdots\partial x_d^{m_d}) f(\bc)$. For $r\geq 1$, the recursive definition of the $r$-order differentiability of $f$ states that $f$ is $r$ times differentiable at $\bc$, if it is $(r-1)$ times differentiable in a neighborhood of $\bc$ and all the partial derivatives $\partial^{\bbm}f$, with $|\bbm|=r-1$, are differentiable at $\bc$. In this case, we are able to write the $r$th order Taylor approximation of $f$ as
\begin{equation}\label{eq:difr}
f(\bc+\bm u)=\sum_{|\bbm|\leq r}\frac{\partial^{\bbm}f(\bc)}{\bbm!}\bm u^{\bbm}+\re_{\bc}(\bm u)
\end{equation}
with $\re_{\bc}(\bm u)/\|\bm u\|^r\to0$ as $\bm u\to{\bm 0}$. See \cite{AGV90} for a thorough discussion about the minimal assumptions necessary for Equation~\eqref{eq:difr} to hold.

Whilst multi-indices offer a concise expression of higher order Taylor approximations, they possess several disadvantages: (i) the Taylor approximation terms, expressed as a multi-index summation, lack of an algebraic representation as a vector, a matrix or any other mathematical object which constitutes a basis for an algebra of differentials, (ii) the concision is not maintained if the infinitesimal element $\bm u$ has the form $\bm u={\bf U}\bz$ for a fixed vector $\bz$ and an infinitesimal matrix $\bf U$, and (iii) perhaps most importantly, there are no results that guarantee the uniqueness of these expressions. We assert that, on the other hand, vectorized derivatives, due to their uniqueness and algebraic properties, are indeed a feasible candidate upon which to build a differential analysis framework.

In Section~\ref{sec:prelim}, we exhibit the mathematical preliminaries required for our investigations into vectorized higher order derivatives. In Section~\ref{sec:ident}, our main results for the identification of differentials and derivatives for scalar- and  vector-valued functions are presented. In Section~\ref{sec:chain}, we extend the basic identification results to rules for the derivatives of the product and the composition of functions. In Section~\ref{sec:example}, we provide concrete examples to illustrate the results from the two previous sections. In Section~\ref{sec:connect}, the connections with some existing results are elaborated. We end with some concluding remarks.

\section{Mathematical preliminaries} \label{sec:prelim}

In the exploration of the broader question of the appropriate form of higher order derivatives of matrix-valued functions, \citet{Mag10} reiterates compelling reasons to define the derivatives of vector-valued functions as a matrix, over alternative forms as a tensor or as a vector, according to the examples presented. Nonetheless, \citet{Pol85} asserts the advantages of the tensor form and \citet{Hol96a} of the vectorized form in other situations.

Whilst the form of the derivative may appear to be an inconsequential theoretical detail, it turns out that the matrix/tensor form of the derivative was one of the key obstacles to solving some important applicative problems; for example, the expression of explicit formulas for moments of arbitrary order of the multivariate Gaussian distribution \citep{Hol88} or the  analysis of general kernel smoothers \citep{CD18} were solved using vectorized derivatives.

Our proposed approach for the analysis of higher order derivatives combines the vectorized form of \citet{Hol96a} with the differential/derivative identification espoused by \citet{MN19}.
We begin with a definition of the required notations for our framework.
For a matrix $\mat A$, denote the $r$th Kronecker power of $\mat A$ as $$\mat A^{\otimes r}=\bigotimes_{i=1}^r\mat A= \overbrace{\mat A\otimes{\cdots}\otimes
\mat A}^{r \ {\rm matrices}}.$$
If $\mat A\in\mathcal M_{m\times n}$ (i.e., $\mat A$ is
a matrix of order $m\times n$) then $\mat A^{\otimes r}\in\mathcal
M_{m^r\times n^r}$; we adopt the convention $\mat A^{\otimes 1}=\mat A$ and $\mat A^{\otimes 0}=1\in\mathbb R$. We also adopt the convention that a vector $\bx = (x_1,\dots, x_d)$ is a column vector. So the derivative operator with respect to the free vector variable $\bx$
is denoted as
$$\D =\frac{\partial}{\partial \bx} = \begin{bmatrix} \dfrac{\partial}{\partial x_1} \\ \vdots \\ \dfrac{\partial}{\partial x_d} \end{bmatrix}$$
and it is a column vector like $\bx$.

Let $f\colon\mathbb R^d\to\mathbb R$ be a scalar-valued function of a $d$-dimensional vector
variable. For an arbitrary non-negative integer $r$, we  consider the object $\D^{\otimes r}f(\bx)\in\mathbb R^{d^r}$ as the $r$th derivative of $f$ at $\bx$. This is a vector containing all the partial derivatives of order $r$ of $f$ at
$\bx$, arranged in a convenient layout as defined by the formal Kronecker power of $\D$.
By the `formal Kronecker power', we mean the product of the differential operator with itself which is obtained using the common notational convention that
$(\partial/\partial x_i) (\partial/\partial x_j) = \partial^2/(\partial x_i \partial x_j) = \partial^2/(\partial x_j \partial x_i)= (\partial/\partial x_j) (\partial/\partial x_i)  $
for all $i,j$. {This commutativity is always guaranteed for a sufficiently regular $f$.} Then we are able to  write formally $$\D^{\otimes r}f (\bx) = \frac{\partial^r f(\bx)}{(\partial\bx)^{\otimes r}}.$$
Hence, the $r$th derivative of $f$ is represented as a vector of length $d^r$,
and not an $r$-fold tensor array or a matrix.

The gradient of $f$ is $\D^{\otimes 1}f =\D f$ so there is no change from the usual derivative here. To observe a difference, we compute explicitly the second derivative. The vectorized Hessian operator is
$$
\D^{\otimes 2} = \frac{\partial^2}{(\partial\bx)^{\otimes 2}}
= \begin{bmatrix} \dfrac{\partial}{\partial x_1} \\ \vdots \\ \dfrac{\partial}{\partial x_d} \end{bmatrix} \otimes \begin{bmatrix} \dfrac{\partial}{\partial x_1} \\ \vdots \\ \dfrac{\partial}{\partial x_d} \end{bmatrix}
= \begin{bmatrix} \dfrac{\partial^2}{\partial x_1^2} \\ \vdots\\ \dfrac{\partial^2}{\partial x_1 \partial x_d} \\ \vdots \\ \dfrac{\partial^2}{\partial x_d \partial x_1} \\ \vdots \\ \dfrac{\partial^2}{\partial x_d^2} \end{bmatrix} =
\begin{bmatrix} \dfrac{\partial^2}{\partial x_1^2} \\ \vdots\\ \dfrac{\partial^2}{\partial x_d \partial x_1} \\ \vdots \\ \dfrac{\partial^2}{\partial x_1 \partial x_d} \\ \vdots \\ \dfrac{\partial^2}{\partial x_d^2} \end{bmatrix},
$$
whereas the usual Hessian operator is
\begin{align*}
\Hessian = \dfrac{\partial^2}{ \partial \bx \partial \bx^\top} = \D \D^\top
&= \begin{bmatrix} \dfrac{\partial^2}{\partial x_1^2} & \dots & \dfrac{\partial^2}{\partial x_1 \partial x_d} \\  \vdots & & \vdots \\ \dfrac{\partial^2}{\partial x_d \partial x_1} & \dots & \dfrac{\partial^2}{\partial x_d^2}\end{bmatrix}
 = \begin{bmatrix} \dfrac{\partial^2}{\partial x_1^2} & \dots & \dfrac{\partial^2}{\partial x_1 \partial x_d} \\  \vdots & & \vdots \\ \dfrac{\partial^2}{\partial x_1 \partial x_d} & \dots & \dfrac{\partial^2}{\partial x_d^2}\end{bmatrix}.
\end{align*}
Therefore the Hessian $\Hessian f$ is such that $\vec\Hessian f =\D^{\otimes 2}f$, where the vec operator transforms a matrix into a vector by stacking its columns underneath each other.

A vectorized form can also be used to express the derivatives of a vector-valued function. If $\bbf \colon \mathbb R^d\to\mathbb R^p$ is a vector-valued function of a vector variable with components $\bbf =(f_1,\dots,f_p)$, then we formally write the $r$th derivative of $\bbf$ at $\bx$ as
$$\D^{\otimes r} \bbf(\bx) = \begin{bbmatrix} \D^{\otimes r} f_1(\bx) \\ \hline \vdots \\ \hline\D^{\otimes r}f_p(\bx)\end{bbmatrix}.$$
Thus $\D^{\otimes r} \bbf$ is a $pd^r$-vector, i.e., we also arrange all partial derivatives in a vector form. For $r=1$, this can be compared to a more traditional, matrix layout: the usual Jacobian matrix ${\mathsf J} \bbf$ is an arrangement of the gradients of the component functions where
$${\mathsf J} \bbf (\bx) =\begin{bbmatrix} \D^\top f_1(\bx) \\ \hline \vdots \\ \hline\D^\top f_p(\bx)\end{bbmatrix}\in \mathcal{M}_{p\times d}.$$
This implies that the first vectorized derivative of $\bbf$, as a column vector of stacked gradient functions, satisfies $\D \bbf = \vec {\mathsf J}^\top \bbf$, echoing the relationship between the vectorized second  derivative of a scalar-valued function with its Hessian matrix, $\D^{\otimes 2} f = \vec \Hessian^\top f = \vec \Hessian f $.

If we restrict ourselves to examining the first and second derivatives, then there is little gain with the vectorized formulation over the traditional formulation of treating the gradient as a vector/matrix and the Hessian as a matrix.
At first glance this configuration of vectorized derivatives may even appear to be a counter-productive arrangement since it breaks the structure of matrix/tensor form of the derivative by rearranging them into a vector. On the other hand, vectorization ensures that we can proceed from the first to the second and to subsequent derivatives without having to change from vector to matrix to tensor. Moreover, for an $r$-times vector-valued differentiable function $\bbf$, vectorization leads to an intuitive, iterative formula for the evaluation of an increment in the derivative order as
$$
\D^{\otimes r} \bbf=\D(\D^{\otimes r-1}\bbf).
$$
This internal consistency affords us many conceptual simplifications which facilitate important advances in higher order differential analysis, which were not able to be treated using the multi-index representations of higher order derivatives.  The most fundamental of these is the existence and uniqueness of the identification between differentials and derivatives.

\section{Identification theorems for higher order differentials} \label{sec:ident}

\subsection{Scalar-valued functions}

We begin with a scalar-valued function $f\colon\mathbb R^d\to\mathbb R$, which we suppose to be $r$-times differentiable at $\bc\in\mathbb R^d$. A common compact notation for its $k$th order partial derivatives is
$$\mathsf D^k_{i_1\cdots i_k}f(\bc)=\frac{\partial^k}{\partial x_{i_1}\cdots\partial x_{i_k}}f(\bc)$$
for any $k\leq r$ and $i_1,\dots,i_k\in\{1,\dots,d\}$. Using this, the $r$th order differential of $f$ at $\bc$ can be expressed as a symmetric $r$-linear form:
\begin{equation}\label{drf}
\diff^rf(\bc;\bu)=\sum_{i_1,\dots,i_r=1}^du_{i_1}\cdots u_{i_r}\mathsf D^r_{i_1\cdots i_r}f(\bc)
\end{equation}
for $\bu=(u_1,\dots,u_d)\in\mathbb R^d$ (\citealp[][p.~193]{F80}; \citealp[][p.~389]{Sch17}). This is another instance of the multi-index notation, since we can rewrite Taylor's theorem in Equation~\eqref{eq:difr} as
$$ f(\bc+\bu)=f(\bc)+\sum_{k=1}^r \sum_{i_1,\dots,i_k=1}^d u_{i_1}\cdots u_{i_k}\D^k_{i_1\cdots i_k}f(\bc) + \re_{\bc}(\bu).$$
Nevertheless, each term of the sum in Equation~\eqref{drf} involves multiplying a certain $r$th order partial derivative by the corresponding coordinates of $\bu$, so the whole sum can be more concisely expressed using a vectorized derivative as $\diff^rf(\bc;\bu)=\D^{\otimes r}f(\bc)^\top\bu^{\otimes r}$.

Apart from concision, a further advantage of this vectorized representation is that it allows for more general forms for the infinitesimal element to be treated easily by applying the usual algebraic properties of the Kronecker product. For example, if $\bu={\bf U}\bz$ for a fixed vector $\bz$ and an infinitesimal matrix ${\bf U}$, then
$\diff^rf(\bc;{\bf U}\bz)=\D^{\otimes r}f(\bc)^\top({\bf U}\bz)^{\otimes r}=\D^{\otimes r}f(\bc)^\top{\bf U}^{\otimes r}\bz^{\otimes r}.$
The compactness of this expression, together with the explicit separation of the infinitesimal part from the other components, cannot be achieved by a multi-index representation. The ability to isolate the infinitesimals plays a key role in the development of new differential analysis results, such as the formulating the uniqueness properties of the following identification theorems.

To establish uniqueness, the symmetrizer matrix plays a crucial role. It was introduced by \cite{Hol85} with the aim of obtaining a symmetrization of the Kronecker product. The symmetrizer matrix $\S_{d,r}$ is implicitly defined by the property that, for any choice of $r$ vectors $\bv_1,\dots,\bv_r\in\mathbb R^d$, it holds that $r!\,\S_{d,r}(\bv_1\otimes\cdots\otimes\bv_r)=\sum_{\sigma\in\mathcal P_r}(\bv_{\sigma(1)}\otimes\cdots\otimes\bv_{\sigma(r)}),$ where $\mathcal P_r$ stands for the set of all permutations of $\{1,\dots,r\}$. It makes the Kronecker product symmetric in the sense that $\S_{d,r}(\bv_1\otimes\cdots\otimes\bv_r)=\S_{d,r}(\bv_{\tau(1)}\otimes\cdots\otimes\bv_{\tau(r)})$ for any permutation $\tau\in\mathcal P_r$.
Thus, $\S_{d,r}$ is a matrix of order $d^r\times d^r$ and it has the following explicit form \citep[see, for example,][]{Sch03}
$$\S_{d,r}=\frac1{r!}\sum_{i_1,\dots,i_r=1}^d\;\sum_{\sigma\in\mathcal
P_r}\bigotimes_{\ell=1}^r\be_{i_\ell}\be_{i_{\sigma(\ell)}}^\top=\frac1{r!}\sum_{i_1,\dots,i_r=1}^d\;\Big(\bigotimes_{\ell=1}^r\be_{i_\ell}\Big)\Big\{\sum_{\sigma\in\mathcal
P_r}\Big(\bigotimes_{\ell=1}^r\be_{i_{\sigma(\ell)}}\Big)\Big\}^\top$$
where $\be_i$ is the $i$th column of the $d\times d$ identity matrix $\bI_d$.
This expression reveals that the computation of $\S_{d,r}$ can be a complex and time-consuming task in practice, especially for large values of $d$ and/or $r$, though we note that  \cite{CD15} developed efficient recursive algorithms to alleviate this problem.

Our first main result is an identification theorem for differentials of arbitrary order with respect to the vectorized derivative, with a corresponding level of uniqueness, for a scalar-valued function of a vector variable.

\begin{theorem}[Scalar-valued identification]\label{thm:ident}
Let the function $f\colon\mathbb R^d\to\mathbb R$ be $r$-times differentiable at $\bc$.
\begin{enumerate}[label=(\roman*)]
\item If $\bu \in \mathbb{R}^d$ then the $r$th order differential of $f$ at $\bc$ with increment $\bu$ is given by
$\diff^rf(\bc;\bu)=\mathsf D^{\otimes r}f(\bc)^\top\bu^{\otimes r}$.
\item If $\ba\in\mathbb R^{d^r}$ satisfies $\diff^rf(\bc;\bu)=\ba^\top\bu^{\otimes r}$ for all $\bu\in\mathbb R^d$, then $\mathsf D^{\otimes r}f(\bc)=\S_{d,r}\ba$.
\end{enumerate}
\end{theorem}

Although rarely expressed with vectorized derivatives, Theorem~\ref{thm:ident}(i) is already known, as is shown above. For Theorem~\ref{thm:ident}(ii), it suffices to establish that $\ba^\top\bu^{\otimes r}=0$ for all $\bu\in\mathbb R^d$ if and only if $\S_{d,r}\ba=0$. This result, labeled as Lemma~\ref{lem:ident-zero}, is stated and proved in Appendix~\ref{app:ident}.

For $r=1$, since $\S_{d,1}=\bI_d$, Theorem~\ref{thm:ident} agrees almost exactly with the first identification theorem of \citet[p.~96]{MN19} except that the latter identifies the first derivative as a row vector (the Jacobian matrix) rather than a column vector as we do. For $r=2$, Theorem~\ref{thm:ident} agrees with the second identification theorem of \citet[p.~119]{MN19}, except that we express it as vectorized derivative whilst \citeauthor{MN19} express it as a Hessian matrix. These authors state that if $\mat A \in \mathcal{M}_{d\times d}$ satisfies $\diff^2f(\bc;\bu)=\bu^\top\mat A\bu$ then the Hessian matrix is identified as $\mathsf Hf(\bc)=(\mat A+\mat A^\top)/2$. Using Theorem~\ref{thm:ident}, since $\diff^2f(\bc;\bu)=\bu^\top\mat A\bu= \ba^\top\bu^{\otimes 2}$ for $\ba=\vec\bA$, then that yields $\D^{\otimes 2} f(\bc) = \S_{d,2}\ba=\vec \{(\mat A+\mat A^\top)/2\}$, by the properties of the symmetrizer matrix \citep[Example~2.1]{Hol96a}; so exactly the same conclusion is reached, because $\D^{\otimes 2} f(\bc)=\vec\mathsf Hf(\bc)$. Importantly, for $r>2$ \citet{MN19} contains no further identification results for these higher order differentials, whereas Theorem~\ref{thm:ident} is valid for differentials of an arbitrary order.

In terms of uniqueness, since $\S_{d,1}=\bI_d$, it is a strict uniqueness for the first order. On the other hand, it is uniqueness-after-symmetrization (i.e. pre-multiplication by the symmetrizer matrix $\S_{d,r}$) for the second and higher orders. \citeauthor{MN19} assert that the second derivative should be identified from $\diff^2 f(\bc;\bu)= \bu^\top \bA \bu$ with its symmetrized version $(\bA + \bA^\top)/2$, rather than $\bA$ on its own, even though there infinitely many matrices that are different to $\bA$ but which yield the same  symmetrized sum. Their reasoning is equivalent to the pre-multiplication by $\S_{d,2}$ in Theorem~\ref{thm:ident}, since $\S_{d,2}\vec \bA=\vec \{(\mat A+\mat A^\top)/2\}$.

However, for $r>2$ there is no simple sequence of elementary matrix operations that can reproduce the action of the symmetrizer matrix, and the ensuing combinatorial explosion means that keeping track of which mixed partial derivatives are identical by construction, say in the multi-index representation, quickly becomes unwieldy. So the explicit exclusion of the symmetrizer matrix in the previously existing results, based on the behavior for the particular first and second order identification theorems, explains in part the hitherto lack of identification results for higher order differentials.

\subsection{Vector-valued functions}

Our next goal is to extend the identification in Theorem~\ref{thm:ident} to a vector-valued function $\bbf \colon\mathbb R^d \to \mathbb R^p$. Recall that when $\bbf$ has components $(f_1,\dots,f_p)$, the $r$th order differential at $\bc$ is a function $\diff^r \bbf(\bc;\cdot)\colon\mathbb R^d\to\mathbb R^p$ defined as
$\diff^r \bbf(\bc;\cdot)=\big(\diff^r f_1(\bc;\cdot),\dots,\diff^r f_p(\bc;\cdot)\big)$. A component-wise application of Theorem~\ref{thm:ident}(i) yields
$$\diff^r \bbf(\bc;\bu)=
\begin{bbmatrix} (\bu^{\otimes r})^\top \mathsf D^{\otimes r}f_1(\bc)\\ \hline
\vdots \\ \hline
(\bu^{\otimes r})^\top \mathsf D^{\otimes r}f_p(\bc) \end{bbmatrix} =\{ \bI_p \otimes (\bu^\top)^{\otimes r} \} \D^{\otimes r} \bbf(\bc)$$
where the last equality follows from reasoning as in \citet[Section~5.9]{CD18}.
An alternative expression of the $r$th order differential is
$$\diff^r \bbf(\bc;\bu)=\begin{bbmatrix} \mathsf D^{\otimes r}f_1(\bc)^\top\\ \hline
\vdots \\ \hline
\mathsf D^{\otimes r}f_p(\bc)^\top \end{bbmatrix}\bu^{\otimes r}=\{\vec_{d^r,p}^{-1} \D^{\otimes r}\bbf(\bc)\}^\top\bu^{\otimes r},$$
where $\vec_{m,n}^{-1}$ denotes the inverse of the isomorphism $\vec\colon\mathcal M_{m\times n}\to\mathbb R^{mn}$, as evaluated in the following lemma.

\begin{lemma}[Inverse vector operator]\label{lem:invvec}
The inverse of the isomorphism $\vec\colon\mathcal M_{m\times n}\to\mathbb R^{mn}$ is given by $\vec_{m,n}^{-1}(\ba)=\{(\vec^\top\bI_n)\otimes\bI_m\}(\bI_n\otimes \ba)$ for $\ba\in\mathbb R^{mn}$.
\end{lemma}
The proof is in Appendix \ref{app:ident-vec}.
\color{black}
Lemma \ref{lem:invvec} allows us to write further
$$\diff^r \bbf(\bc;\bu)=\{\vec_{d^r,p}^{-1} \D^{\otimes r}\bbf(\bc)\}^\top\bu^{\otimes r}=\{\bI_p\otimes\D^{\otimes r}\bbf(\bc)^\top\}\{(\vec\bI_p)\otimes\bI_{d^r}\}\bu^{\otimes r}.$$
Along with the previous formula,
these three expressions for the $r$th order differential serve different purposes. The first one $\{ \bI_p \otimes (\bu^\top)^{\otimes r} \} \D^{\otimes r} \bbf(\bc)$ is minimal in the sense that it involves the least number of elementary operations. The second one $\{\vec_{d^r,p}^{-1} \D^{\otimes r}\bbf(\bc)\}^\top\bu^{\otimes r}$  separates out the infinitesimal $\bu^{\otimes r}$, and is the most easily identifiable as the generalization of the differential for a scalar function, though this requires the introduction of the inverse vector operator. The third one $\{\bI_p\otimes\D^{\otimes r}\bbf(\bc)^\top\} \{(\vec\bI_p)\otimes\bI_{d^r}\}\bu^{\otimes r}$  is a compromise of these two where a separation of the infinitesimal is attained without the inverse vector operator, but with more involved operations.

\begin{theorem}[Vector-valued identification] \label{thm:ident-vec}
Let the function $\bbf \colon\mathbb R^d \to \mathbb R^p$ be $r$-times differentiable at $\bc$.
\begin{enumerate}[label=(\roman*)]
\item If $\bu \in \mathbb{R}^{d}$ then the $r$th order differential of $\bbf$ at $\bc$ with increment $\bu$ is given by
$\diff^r \bbf(\bc;\bu) = \{ \bI_p \otimes (\bu^\top)^{\otimes r} \} \D^{\otimes r} \bbf(\bc)=\{\vec_{d^r,p}^{-1} \D^{\otimes r}\bbf(\bc)\}^\top\bu^{\otimes r}$.
\item
If $\ba\in\mathbb R^{pd^r}$ satisfies $\diff^r \bbf(\bc;\bu)=(\vec_{d^r,p}^{-1}\ba)^\top\bu^{\otimes r}$ for all $\bu\in\mathbb R^d$, then $\mathsf D^{\otimes r} \bbf(\bc) = (\bI_p \otimes \S_{d,r}) \ba$. If $\bA \in \mathcal{M}_{d^r \times p}$ satisfies $\diff^r \bbf(\bc;\bu)=\bA^\top\bu^{\otimes r}$ for all $\bu\in\mathbb R^d$, then $\mathsf D^{\otimes r} \bbf(\bc) = (\bI_p \otimes \S_{d,r}) \vec \bA = \vec (\S_{d,r} \bA)$.
\end{enumerate}
\end{theorem}

Theorem~\ref{thm:ident-vec}(i) is shown above. The proof of  Theorem~\ref{thm:ident-vec}(ii) is deferred to Appendix~\ref{app:ident-vec}.

Observe that $\bA = \vec_{d^r,p}^{-1}\ba=\ba$ for $p=1$, which ensures that, for the case of a scalar-valued function, Theorem~\ref{thm:ident-vec} reduces to Theorem~\ref{thm:ident}.
Furthermore, observe that $\S_{d,r}=\bI_d$ for $r=1$, which implies that the symmetrizer matrix in effect is not involved in the identification of the first derivative, since if $\diff \bbf(\bc;\bu)=\mat A^\top\bu$ for some $\bA\in\mathcal M_{d\times p}$ then $\mathsf D\bbf(\bc) = \ba=\vec\bA$.

Theorems~\ref{thm:ident} and \ref{thm:ident-vec} are useful to obtain the $r$th order derivative by iterating from the first differential, which may require considerable matrix algebra to isolate the $r$-fold Kronecker product of the infinitesimal $\bu^{\otimes r}$.
The following theorem provides an alternative with the identification of the $r$th derivative from the differential of the $(r-1)$th order derivative.

\begin{theorem}[Iterative identification] \label{thm:ident-iter}
Let $\bbf\colon\mathbb R^d\to\mathbb R^p$ be a function that is $r$-times differentiable at $\bc$, for some $r >1$. Further suppose that its $(r-1)$th derivative, $\D^{\otimes (r-1)} \bbf$, has been already obtained.
If $\mat B\in\mathcal{M}_{d \times pd^{r-1}}$ satisfies $\diff \{\D^{\otimes (r-1)} \bbf\}(\bc;\bu) =\mat B^\top \bu$ for all $\bu\in\mathbb R^d$, then $\mathsf D^{\otimes r} \bbf(\bc)= \vec \mat B$.
\end{theorem}

The proof is in Appendix~\ref{app:ident-iter}.

If we have that $\diff^r \bbf(\bc;\bu)=\mat A^\top\bu^{\otimes r}$ and
$\diff \{\D^{\otimes (r-1)} \bbf\}(\bc;\bu)=\mat B^\top\bu$ for all $\bu\in\mathbb R^d$,
then Theorems~\ref{thm:ident-vec} and \ref{thm:ident-iter} imply that $\vec(\S_{d,r} \bA) = \D^{\otimes r} \bbf(\bc) = \vec \bB$, although $\bA \neq \bB$ in general since $\bA \in \mathcal{M}_{d^r\times p}$ and $\bB \in \mathcal{M}_{d \times pd^{r-1}}$.  Therefore, $\bA$ and $\bB$ must contain the same elements but in a different layout. It is the joint action of the vectorization and the symmetrizer matrix that facilitates their re-arrangement into a common form $\D^{\otimes r} \bbf(\bc)$.

From Theorem~\ref{thm:ident-iter}, to obtain the $r$th derivative we are only required to compute a first order differential as a product of matrix and a single $d$-vector infinitesimal $\bu$ at each iteration, which can be easier to compute than the $r$th order differential as a product of a matrix and a $r$-fold Kronecker product of the $d$-vector infinitesimal $\bu^{\otimes r}$ required in Theorem~\ref{thm:ident-vec}.

\subsection{Matrix-valued functions and functions of a matrix variable}

We end with a discussion on our proposition for the derivative of a matrix-valued function and a function of a matrix variable. Let $\mat F\colon\mathbb{R}^d \to \mathcal{M}_{p\times q}$ be a matrix-valued function of a vector variable. Following on from our treatment of vector-valued functions, it is straightforward to apply the identification in Theorem~\ref{thm:ident-vec} to $\vec\mat F\colon\mathbb{R}^d \to \mathbb{R}^{pq}$, since it is a vector-valued function of a vector variable.

Thus the outstanding question is the analysis of functions of a matrix variable $\bX$. Even if it appears initially to be most intuitive to define derivatives with respect to $\bX$, e.g. as exposited in \citet[Section~1.4]{KvR05} and \cite{Mag10}, these same authors in their respective papers subsequently argue that this is not desirable for many reasons. Instead, they propose to also vectorize the free variable, that is to analyze $\vec\mat F (\bX)$ with respect to $\vec \bX$. Whilst they restrict themselves to the first and second order derivatives, in our case, we can appeal to Theorems~\ref{thm:ident} and \ref{thm:ident-vec} for arbitrary order derivatives.

Although it is out of scope of this paper to settle definitively this difficult question of derivatives with respect to matrix variables, we highlight that our vectorizing approach offers systematic solutions to the key questions of: (i) how to define the dimensions of the derivatives, and (ii) how to identify higher order differentials with their derivatives. Let the $(i,j)$th component function of $\mat F$ be $f_{ij}$ for $i=1,\dots,p$, $j=1,\dots,q$, and $\bX\in\mathcal M_{c\times d}$. Then, the $r$th derivative of $\vec\mat F$ with respect to $\vec\bX$ is defined to be the vector
\begin{equation}\label{eq:DvecF}
\D^{\otimes r}\vec\mat F(\mat X)=\begin{bbmatrix} \mathsf D^{\otimes r}f_{11}(\mat X)\\ \hline
\vdots \\ \hline
\mathsf D^{\otimes r}f_{pq}(\mat X)\end{bbmatrix}\in\mathbb R^{pqc^rd^r},
\end{equation}
where $\mathsf D^{\otimes r}f_{ij}(\mat X)=\partial^rf_{ij}(\mat X)/(\partial\vec\bX)^{\otimes r}\in\mathbb R^{c^rd^r}$ for each $i,j$.

Thus by enumerating the possible combinations in Equation~\eqref{eq:DvecF}, our answers to the former two questions are summarized in Table~\ref{tab:ident}. This table contains the identifications for an arbitrary order $r$ for all the combinations a scalar $f \in \mathbb{R}$, vector $\bbf\in \mathbb{R}^p $ and matrix-valued $\mat F\in \mathcal{M}_{p\times q}$ function of a scalar $x \in \mathbb{R}$, vector $\bx \in \mathbb{R}^d$ and matrix $\bX \in \mathcal{M}_{c\times d}$ variable.  Following the notational convention of \citet{MN19}, we denote the infinitesimal as $\diff \bx$ etc. in Table~\ref{tab:ident} rather than $\bu$ as in the theorem statements.

\begin{table}[!ht]
\centering
\setlength{\tabcolsep}{2pt}
\begin{tabular}{
  c
  *{2}{>{$}r<{$}@{}>{${}={}$}c@{}>{$}l<{$}}
  c
}
Function & \multicolumn{3}{c}{Differential} & \multicolumn{3}{c}{Derivative} & \multicolumn{1}{c}{Dimension} \\ \hline \hline
  $f(x)$
& \diff^r f(x) && a (\diff x)^r
& \D^{\otimes r} f(x) && a
& $\mathbb{R}$ \\
  $\bbf(x)$
& \diff^r \bbf(x) && \ba (\diff x)^r
& \D^{\otimes r} \bbf(x) && \ba
& $\mathbb{R}^p$ \\
  $\mat F(x)$
& \diff^r \vec \mat F(x) && \ba (\diff x)^r
& \D^{\otimes r} \vec \mat F(x) && \ba
& $\mathbb{R}^{pq}$ \\ \hline
  $f(\bx)$
& \diff^r f(\bx) && \ba^\top (\diff\bx)^{\otimes r}
& \D^{\otimes r} f(\bx) && \S_{d,r} \ba
& $\mathbb{R}^{d^r}$ \\
  $\bbf(\bx)$
& \diff^r \bbf(\bx) && (\vec_{d^r,p}^{-1}\ba)^\top (\diff\bx)^{\otimes r}
& \D^{\otimes r} \bbf(\bx) && (\bI_p \otimes \S_{d,r}) \ba
& $\mathbb{R}^{pd^r}$ \\
  $\mat F(\bx)$
& \diff^r \vec\mat F(\bx) && (\vec_{d^r,pq}^{-1}\ba)^\top (\diff\bx)^{\otimes r}
& \D^{\otimes r} \vec \mat F(\bx) && (\bI_{pq} \otimes \S_{d,r}) \ba
& $\mathbb{R}^{pq d^r}$\\
\hline
  $f(\bX)$
& \diff^r f(\bX) && \ba^\top (\diff\vec\bX)^{\otimes r}
& \D^{\otimes r} f(\bX) && \S_{cd,r}\ba
& $\mathbb{R}^{c^rd^r}$ \\
  $\bbf(\bX)$
& \diff^r \bbf(\bX) && (\vec_{c^rd^r,p}^{-1}\ba)^\top (\diff\vec\bX)^{\otimes r}
& \D^{\otimes r} \bbf(\bX) && (\bI_p \otimes \S_{cd,r}) \ba
& $\mathbb{R}^{p c^rd^r}$ \\
  $\mat F(\bX)$
& \diff^r \vec \mat F(\bX) && (\vec_{c^rd^r,pq}^{-1}\ba)^\top (\diff\vec\bX)^{\otimes r}
& \D^{\otimes r} \vec \mat F(\bX) && (\bI_{pq} \otimes \S_{cd,r}) \ba
& $\mathbb{R}^{pq c^rd^r}$
\end{tabular}
\caption{Higher order identifications, for the functions $f \in \mathbb{R}, \bbf\in \mathbb{R}^p, \mat F\in \mathcal{M}_{p\times q}$, and the variables $x \in \mathbb{R}, \bx \in \mathbb{R}^d, \bX \in \mathcal{M}_{c\times d}$. The first column is the function, the second is the $r$th order differential, the third is the $r$th derivative and the fourth is the dimension of the vectorized derivative. }
\label{tab:ident}
\end{table}

The differentials and derivatives of the vector- and matrix-valued functions in the second and third columns in Table~\ref{tab:ident} are vectors.
In contrast, whilst \citet{KvR05} and \citet{Mag10} also define the differentials as vectors, they insist that matrix-valued derivatives be identified with these vector-valued differentials. For instance,
\citet[][p. 126]{KvR05} express their preference to define the (first) derivative of $\bbf=(f_1,\dots,f_p)$ as
$$
\dfrac{\partial \bbf^\top}{\partial \vec \bX}
= \bbf^\top \otimes \dfrac{\partial}{\partial \vec \bX}
= \left[ \begin{array}{@{}c|c|c@{}}
\dfrac{\partial f_1}{\partial \vec \bX}
& \dots
& \dfrac{\partial f_p}{\partial \vec \bX}
\end{array}\right]
\in \mathcal{M}_{cd \times p},
$$
whereas \citet{Mag10} employs the transpose of this arrangement
$$
\dfrac{\partial \bbf}{\partial \vec^\top \bX}
= \bbf\otimes\left(\frac{\partial}{\partial \vec \bX}\right)^\top
= \left(\dfrac{\partial \bbf^\top}{\partial \vec \bX}\right)^\top
\in \mathcal{M}_{p \times cd}.
$$
Our vectorized derivative in Equation~\eqref{eq:DvecF} is related to the \citeauthor{KvR05} arrangement because it is the vectorization of the former:
$$
\D \bbf = \dfrac{\partial \bbf}{\partial \vec \bX}
= \bbf \otimes \dfrac{\partial}{\partial \vec \bX} = \vec \left( \dfrac{\partial \bbf^\top}{\partial \vec \bX} \right) \in \mathbb{R}^{pcd},
$$
since we have $\ba \otimes \bb = \vec (\ba^\top \otimes \bb)$ for any vectors $\ba, \bb$. On the other hand, since $\D \bbf$ is the vectorization of the transpose of \citeauthor{Mag10} arrangement, it also retains the relationship with the vectorization of the transposed Jacobian for a function of a vector variable.

Hence, Table~\ref{tab:ident} with $r=1$ is essentially the same as Table 9.2 in \citet{MN19} for the first order identification. For the second order identification, since these matrix-valued derivative forms are composed of arrangements of blocks of matrices, these authors have only been able to define an identification with a second order derivative for a scalar-valued function $f$ as $\partial f/[(\partial \vec \bX) (\partial \vec^\top \bX)]$ or $\partial f/[(\partial \vec^\top \bX) (\partial \vec \bX)]$. Thus for the second order identification, Table 10.3 in \citet{MN19} contains only the rows of Table~\ref{tab:ident} with scalar $f$ for $r=2$.

Whilst these authors establish further important properties for an algebra for their first derivatives, e.g. rules for the product and the composition of two functions,  we demonstrate next that vectorized derivatives facilitate a systematic definition of these rules for arbitrary order derivatives.

\section{Product and chain rules for vectorized higher order derivatives} \label{sec:chain}

We examine the product and chain rules for higher order derivatives. The product rule for the multiplication by constants is the easiest to establish since the differential operator is a linear operator. The product rule for the product of two scalar-valued functions, also known as the general Leibniz rule, is well known, though the case for the product of vector-valued functions remains largely unexamined.
The chain rule for the composition of two functions is also known as the Fa\`{a} di Bruno's formula, see \citet[Chapter~4]{Ave97} and \citet{CS96} for their treatment of scalar-valued functions. Again, the case for vector-valued functions remains largely unknown. With our vectorized differential analysis framework we re-cast some of these existing results and develop other hitherto unestablished ones.

For brevity, we provide the results for scalar- and vector-valued functions of vector variables, since the results (i) for scalar variables can be immediately inferred from the results for vector variables, and (ii) for matrix-valued functions and matrix variables, if they are vectorized beforehand, can be immediately inferred from those for vector-valued functions and vector variables.

We begin with some rules for the derivative of a function multiplied by a constant coefficient (i.e. the latter does not involve the free variable). The proof of all the results in this section, which are given in Appendix~\ref{app:prod-chain}, make an extensive use of Theorems \ref{thm:ident} and \ref{thm:ident-vec} to identify the derivatives from the differentials but, again for brevity, we state these results only in terms of derivatives.

\begin{theorem}[Constant multiplication]\label{thm:deriv-const} \,
\begin{enumerate}[label=(\roman*)]
\item Let the function $f\colon\mathbb R^d\to\mathbb R$ be $r$-times differentiable at $\bc$. If $\ba \in \mathbb{R}^q$ is a constant vector, then the $r$th derivative of $\ba f$ at $\bc$ is
$ \D^{\otimes r}(\ba f)(\bc) = \ba \otimes \D^{\otimes r } f(\bc) \in \mathbb{R}^{qd^r}$.

\item Let the function $\bbf\colon\mathbb R^d\to\mathbb R^p$ be $r$-times differentiable at $\bc$. If $\ba \in \mathbb{R}^q$ is a constant vector, then the $r$th derivative of $\ba\otimes\bbf$ at $\bc$ is $\D^{\otimes r} (\ba \otimes \bbf)(\bc) = \ba \otimes \D^{\otimes r } \bbf(\bc) \in\mathbb{R}^{pqd^r}$.
If $\bA \in \mathcal{M}_{q\times p}$ is a constant matrix, then the $r$th derivative of $\bA\bbf$ at $\bc$ is $\D^{\otimes r} (\bA \bbf)(\bc) = (\bA \otimes \bI_{d^r})  \D^{\otimes r} \bbf(\bc)\in\mathbb{R}^{qd^r}$.
\end{enumerate}
\end{theorem}

Theorem~\ref{thm:deriv-const} verifies that $\D^{\otimes r}$ demonstrates an expected behavior under constant multiplication.
From this, we next move onto the derivative of the product of two functions. Whilst the product rule for higher order partial derivatives of the product of two scalar-valued functions is well established as the general Leibniz rule, we establish it here for the vectorized derivative of the Kronecker product of two vector-valued functions.

\begin{theorem}[General Leibniz rule]\label{thm:deriv-prod} \,
\begin{enumerate}[label=(\roman*)]
\item Let the functions $f,g\colon\mathbb R^d\to\mathbb R$ be $r$-times differentiable at $\bc$. Then the $r$th derivative of $f\cdot g$ at $\bc$ is
$$\D^{\otimes r}(f\cdot g)(\bc)=\S_{d,r}\sum_{j=0}^r\binom{r}{j}\D^{\otimes r-j}f(\bc)\otimes\D^{\otimes j}g(\bc) \in \mathbb{R}^{d^r}.$$

\item Let the functions $\bbf\colon\mathbb R^d\to\mathbb R^p, \bg\colon\mathbb R^d\to\mathbb R^q$ be $r$-times differentiable at $\bc$. Then the $r$th derivative of $\bbf\otimes\bg$ at $\bc$ is
$$
\D^{\otimes r}(\bbf \otimes \bg)(\bc) = (\bI_{pq} \otimes \S_{d,r}) \sum_{j=0}^r\binom{r}{j} \vec\big\{\vec^{-1}_{d^{r-j},p}\D^{\otimes r-j}\bbf(\bc)  \otimes \vec^{-1}_{d^{j},q}\D^{\otimes j}\bg(\bc)\big\}\in \mathbb{R}^{pqd^r}.
$$
\end{enumerate}
\end{theorem}

The terms in the latter summation can be simplified by introducing commutation matrices to express the vectorized form of a Kronecker product of matrices in terms of the Kronecker product of the vectorized matrices \citep[][Theorem 3.10]{MN19}, leading to
$$\vec\big\{\vec^{-1}_{d^{r-j},p}\D^{\otimes r-j}\bbf(\bc)  \otimes \vec^{-1}_{d^{j},q}\D^{\otimes j}\bg(\bc)\big\}=(\bI_p \otimes \mat K_{q,d^{r-j}} \otimes \bI_{d^j}) \{\D^{\otimes r-j}\bbf(\bc)\otimes\D^{\otimes j}\bg(\bc)\}.$$
{Commutation matrices are widely utilized in matrix algebra and analysis: one of their most important properties is that an $(m,n)$ order commutation matrix $\K_{m,n} \in \mathcal{M}_{mn \times mn}$ satisfies $\K_{m,n} (\vec \bA) = \vec (\bA^\top)$ for an $m \times n$  matrix $\bA$, see \citet[Chapter~8.6]{Sch17} and \citet[Chapter~3.5]{MN19} for an overview of their properties.}

The statement of the general Leibniz rule for scalar functions of a vector variable is usually expressed for each partial derivative singly with a multi-index notation, e.g. \citet[Lemma~2.6]{CS96} or \citet{Har06}, whereas Theorem~\ref{thm:deriv-prod}(i) offers a concise, global expression containing all the $r$th order partial derivatives. Theorem~\ref{thm:deriv-prod}(ii) extends the Leibniz rule to the Kronecker product of two vector-valued functions.

The last situation that we consider in this section concerns a formula for the higher order derivatives of the composition of two functions, where the composition is defined by $(\bg\circ \bbf)(\bx)=\bg\{\bbf(\bx)\}$. Expressing the derivatives of $\bg\circ \bbf$ in terms of the derivatives of $\bg$ and $\bbf$ will prove to be very useful to obtain some complicated derivatives in a simple way, as we will see in Sections \ref{sec:example} and \ref{sec:connect} below.

First, from the usual chain rule for Jacobian matrices it immediately follows that, for a scalar-valued function $g\colon\mathbb R^p\to\mathbb R$, the first derivative of the composition with $\bbf\colon\mathbb R^d\to\mathbb R^p$ at $\bc$ is $\D(g\circ \bbf)(\bc)=\big[\D g\{\bbf(\bc)\}^\top\otimes\bI_d\big]\D\bbf(\bc)$, and for a vector-valued function $\bg\colon\mathbb R^p\to\mathbb R^q$ the former can be generalized to $\D(\bg\circ\bbf)(\bc)=\big([\vec_{p,q}^{-1}\D \bg\{\bbf(\bc)\}]^\top\otimes\bI_d\big)\D\bbf(\bc)$.


The goal is to derive a formula for the $r$th derivative of $\bg\circ\bbf$ for an arbitrary $r$. The computation of this derivative involves the set $\mathcal J_r=\{\bbm=(m_1,\dots,m_r)\in\mathbb N_0^r\colon\sum_{\ell=1}^r\ell m_\ell=r\}$ containing all the non-negative integer solutions of $1\cdot m_1+2\cdot m_2+\cdots+r\cdot m_r=r$, which can be also expressed as $\mathcal J_r=\bigcup_{k=1}^r\mathcal J_{k,r}$, with $\mathcal J_{k,r}=\{\bbm\in\mathcal J_r\colon|\bbm|=k\}$. Note that \citet[Equation~(11.18)]{VN97} supply a computationally efficient algorithm for enumerating all elements of $\mathcal J_r$. Furthermore, let us denote $\pi_{\bbm}=r!/\prod_{\ell=1}^r[m_\ell!(\ell!)^{m_\ell}]$ for any $\bbm\in\mathcal J_r$.

\begin{theorem}[Fa\`a di Bruno's formula]\label{thm:deriv-comp}
Let the function $\bbf\colon\mathbb R^d\to\mathbb R^p$ be $r$-times differentiable at $\bc$.
\begin{enumerate}[label=(\roman*)]
\item Let $g\colon\mathbb R^p\to\mathbb R$ be $r$-times differentiable at $\bbf(\bc)$. Then the $r$th derivative of $g\circ\bbf$ at $\bc$ is
$$\D^{\otimes r}(g\circ \bbf)(\bc) = \sum_{\bbm\in\mathcal J_{r}}\pi_{\bbm}\big[\D^{\otimes |\bbm|}g\{\bbf(\bc)\}^\top\otimes\S_{d,r}\big]\bigotimes_{\ell=1}^r\{\D^{\otimes \ell}\bbf(\bc)\}^{\otimes m_\ell}\in\mathbb R^{d^r},$$
where $\D^{\otimes |\bbm|} g\{\bbf(\bc)\} = \D^{\otimes |\bbm|}g(\bc') \lvert_{\bc'= \bbf(\bc)}$ denotes the $|\bbm|$th derivative of $g$ evaluated at $\bbf(\bc)$.

\item Let $\bg\colon\mathbb R^p\to\mathbb R^q$ be $r$-times differentiable at $\bbf(\bc)$. Then the $r$th derivative of $\bg\circ\bbf$ at $\bc$ is
\begin{align*}
\D^{\otimes r}(\bg \circ \bbf) (\bc)
&=\sum_{\bbm\in\mathcal J_{r}}\pi_{\bbm} \Big( \big[\vec_{p^{|\bbm|},q}^{-1} \D^{\otimes |\bbm|}\bg\{\bbf(\bc)\} \big]^\top \otimes\S_{d,r} \Big)  \bigotimes_{\ell=1}^r\{\D^{\otimes \ell}\bbf(\bc)\}^{\otimes m_\ell}\in \mathbb{R}^{qd^r}.
\end{align*}
\end{enumerate}
\end{theorem}

{Lemma~\ref{lem:invvec} can be invoked to simplify the inverse vec in Theorem~\ref{thm:deriv-comp}(ii) if required.}
Theorem 2.1 in \citet{CS96} also provides a higher order chain formula, but only for individual partial derivatives using multi-indices, whereas Theorem~\ref{thm:deriv-comp}(i) offers a concise, global expression containing all the $r$th order partial derivatives. Moreover, \citeauthor{CS96} noted that it is highly difficult to obtain their results for $g\circ\bbf$, let alone $\bg \circ \bbf$, and so did not supply the latter, whereas Theorem~\ref{thm:deriv-comp}(ii) follows naturally from Theorem~\ref{thm:deriv-comp}(i).

\section{Some important vector and matrix functions}  \label{sec:example}

To illustrate the applicability of the identification theorems, we consider the derivatives of  some important vector and matrix functions.
Before we tackle the differential analysis, we establish an identity in Lemma~\ref{lem:vec-matrix-prod} which is a useful generalization of the well-known formula $\vec(\bA\bB\bC)=(\bC^\top\otimes\bA)\vec\bB$.

\begin{lemma}\label{lem:vec-matrix-prod}
Consider matrices $\bA \in \mathcal{M}_{m \times n}, \bB, \bC \in \mathcal{M}_{n \times n}, {\mat D} \in \mathcal{M}_{n \times p}$. Then, for $r\geq1$,
$$\vec \{\bA \bB (\bC \bB)^{r-1} {\mat D}\} = \vec \{\bA (\bB \bC)^{r-1} \bB {\mat D}\} = \{{\mat D}^\top \otimes (\vec^\top \bC)^{\otimes r-1} \otimes \bA\}  (\vec \bB)^{\otimes r}$$
where by convention $(\bC\bB)^0 = (\bB\bC)^0 = \bI_n$. For $\bA = \bC = {\mat D}$, the former identity simplifies to
$\vec \{(\bA \bB)^r \bA\} = \vec \{\bA (\bB \bA)^r\} = \{\bA^\top \otimes (\vec^\top \bA)^{\otimes r-1} \otimes  \bA\} (\vec \bB)^{\otimes r}$.
Furthermore, for $\bA = \bC = {\mat D} = \bI_n$, it yields $\vec (\bB^r) = \{\bI_n \otimes (\vec^\top \bI_n)^{\otimes r-1} \otimes \bI_n\} (\vec \bB)^{\otimes r}$.
\end{lemma}

The proof is detailed in Appendix~\ref{app:matrix}. This result is crucial for computing the differentials of the matrix functions that we consider in the sequel, since when it is applied with $\bB = \diff\bX$, then Lemma~\ref{lem:vec-matrix-prod} isolates $(\diff \vec \bX)^{\otimes r}$ which is required to identify an $r$th order derivative.

One of the fundamental rules of real analysis is the derivative of a monomial given by $(\diff^r/\diff x^r) (x^k) = 
\{k!/(k-r)!\} x^{k-r}$ for $r\leq k$. For a vector variable $\bx$ and a matrix variable $\bX$, the equivalent results are given in Theorem~\ref{thm:deriv-monomial}.

\begin{theorem}[Monomial]\label{thm:deriv-monomial}
Consider $1 \leq r \leq k$.
\begin{enumerate}[label=(\roman*)]
\item Let the function $\bbf\colon\mathcal M_{c\times d}\to\mathbb R^{b^k}$ be the $k$-fold Kronecker product $\bbf (\bX)=(\vec\bX)^{\otimes k}$, where $b=cd$. Then the $r$th derivative of $(\vec\bX)^{\otimes k}$ is
$$\D^{\otimes r}\{(\vec\bX)^{\otimes k}\}=(\mat\Gamma_{b,k,r}\otimes\S_{b,r})\big\{(\vec\bX)^{\otimes k-r}\otimes\vec\bI_{b^{r}}\big\}\in\mathbb R^{b^{k+r}},$$
where $\mat\Gamma_{b,k,r}=\prod_{j=0}^{r-1}(\mat\Lambda_{b,k-j}\otimes\bI_{b^{j}})=\mat\Lambda_{b,k}(\mat\Lambda_{b,k-1}\otimes \bI_{b})\cdots(\mat\Lambda_{b,k-r+1}\otimes \bI_{b^{r-1}})\in\mathcal M_{b^{k}\times b^{k}}$, with $\mat\Lambda_{b,k}=\sum_{j=1}^k\mat K_{b^{j},b^{k-j}}\in\mathcal M_{b^{k}\times b^{k}}$ being the sum of commutation matrices.

\item Let the function $\bbf\colon\mathbb R^d\to\mathbb R^{d^k}$ be the $k$-fold Kronecker product
$\bbf (\bx) = \bx^{\otimes k}$. Then the $r$th derivative of $\bx^{\otimes k}$ is
$$\D^{\otimes r} (\bx^{\otimes k}) =(\mat\Gamma_{d,k,r}\otimes\S_{d,r})\big(\bx^{\otimes k-r}\otimes\vec\bI_{d^{r}}\big)\in\mathbb R^{d^{k+r}}.$$

\item Let the function $\mat F \colon\mathcal{M}_{d\times d}\to\mathcal{M}_{d\times d}$ be the $k$th matrix power
$\mat F(\bX) = \bX^k$. Then the $r$th derivative of $\bX^k$ is
\begin{align*}
\D^{\otimes r} \vec (\bX^k)= (\mat\Upsilon_{d,k,r}\otimes\S_{d^2,r})\big\{(\vec\bX)^{\otimes k-r}\otimes\vec\bI_{d^{2r}}\big\} \in \mathbb{R}^{d^{2r+2}},
\end{align*}
where $\mat\Upsilon_{d,k,r}=\{\bI_d \otimes (\vec^\top \bI_d)^{\otimes k-1} \otimes \bI_d\} \mat\Gamma_{d^2,k,r}\in\mathcal M_{d^2\times d^{2k}}$.

\item Let function $f\colon\mathcal{M}_{d\times d}\to\mathbb R$ be the trace of the $k$th matrix power $f(\bX) = \tr \bX^k$. Then the $r$th derivative of $\tr \bX^k$ is
\begin{align*}
\D^{\otimes r} \tr (\bX^k)=
\big[\{(\vec^\top\bI_d) \mat\Upsilon_{d,k,r}\} \otimes\S_{d^2,r} \big] \big\{(\vec\bX)^{\otimes k-r}\otimes\vec\bI_{d^{2r}}\big\}
\in \mathbb{R}^{d^{2r}}.
\end{align*}
\end{enumerate}
\end{theorem}

The proof is given in Appendix~\ref{app:matrix} with the assistance of the identification Theorem~\ref{thm:ident-vec}, and Theorem~\ref{thm:deriv-const} and Lemma~\ref{lem:vec-matrix-prod} to handle the matrix powers in terms of the Kronecker powers.

For a scalar variable we have $\mat\Lambda_{1,k}=k$ and $\mat\Gamma_{1,k,r}=k\cdot(k-1)\cdots(k-r+1)=k!/(k-r)!$, so in that case all the previous derivatives coincide with the usual derivative of a monomial. Furthermore, \citet[p.~186]{MN19} established that $\diff \tr (\bX^k) = k \tr (\bX^{k-1} \diff\bX) = k \vec^\top \{(\bX^\top)^{k-1}\} \diff \vec \bX$, so that $\D \tr (\bX^k) = k \vec \{(\bX^\top)^{k-1}\}$. Whilst this first derivative is considerably simpler than our computation, it encounters intractable difficulties for higher order derivatives. In contrast, our more general formula in Theorem~\ref{thm:deriv-monomial} is well defined for all derivatives of $\tr (\bX^k)$, and for the first derivative we {verify that it is equivalent to the \citeauthor{MN19} formula in the last part of the proof of Theorem~\ref{thm:deriv-monomial}.}


In addition to the matrix monomial, other widely utilized matrix functions are the matrix inverse and the matrix determinant. The first order derivatives, and some higher order differentials (but not the corresponding derivatives), are established by \citet{MN19}: we extend them to a complete description of the $r$th order derivatives.

\begin{theorem}[Matrix inverse] \label{thm:matrix-inv}
Let the function $\mat F \colon \mathcal{M}_{d\times d}  \to \mathcal{M}_{d\times d} $ be the matrix inverse function $\mat F(\bX) = \bX^{-1} $ for all non-singular $\bX$. Then the $r$th derivative of $\bX^{-1}$ is
\begin{align*}
\D^{\otimes r} (\vec \bX^{-1})
&= (-1)^r r! (\bI_{d^2} \otimes \S_{d^2,r}) (\bI_d \otimes \K_{d,d^{2r}}) (\vec \bX^{-1})^{\otimes r+1}
 \in \mathbb{R}^{d^{2r+2}}.
\end{align*}
\end{theorem}

The proof of Theorem~\ref{thm:matrix-inv} is in Appendix \ref{app:matrix}.

For $r=1$, the proof of Theorem \ref{thm:matrix-inv} shows that $\D (\vec \bX^{-1}) = - \vec \{\bX^{-1} \otimes (\bX^{-1})^\top\}$, which agrees with \citet[Theorem~8.2, p.~168]{MN19}. For $r>1$, whilst \citeauthor{MN19} are able to assert that the $r$th order matrix differential as $\diff^r (\bX^{-1}) =  (-1)^r r! (\bX^{-1} \diff \bX)^r \bX^{-1}$, they have no identification result to transform this into a derivative. Theorem~\ref{thm:matrix-inv} establishes this $r$th order derivative via the application of Lemma~\ref{lem:vec-matrix-prod} to compute $\diff^r (\vec \bX^{-1})$ and subsequently the vector-valued identification in Theorem~\ref{thm:ident-vec}.


\begin{theorem}[Matrix determinant] \label{thm:matrix-det}
Consider $r\geq 1$.
\begin{enumerate}[label=(\roman*)]

\item Let the function $f \colon \mathcal{M}_{d\times d} \to \mathbb{R}$ be the log matrix determinant $f(\bX) = \log |\bX|$ for all non-singular $\bX$. Then the $r$th derivative of $\log |\bX|$ is
\begin{align*}
	\D^{\otimes r} \log |\bX|
	&= (-1)^{r-1}(r-1)!\S_{d^2,r}\mat K_{d,d^{2r-1}}(\vec\bX^{-1})^{\otimes r} \in \mathbb{R}^{d^{2r}}.
\end{align*}

\item Let the function $f \colon \mathcal{M}_{d\times d}  \to \mathbb{R}$ be the matrix determinant $f(\bX) = |\bX|$ for all non-singular $\bX$. Then the $r$th derivative of $|\bX|$  is
\begin{align*}
\D^{\otimes r} |\bX|
= |\bX|\S_{d^2,r}\mat\Xi_{d,r} (\vec \bX^{-1})^{\otimes r} \in \mathbb{R}^{d^{2r}}
\end{align*}
where
$$\mat\Xi_{d,r} = \sum_{\bbm\in\mathcal J_{r}}(-1)^{r-|\bbm|}\frac{r!}{\prod_{\ell=1}^r(m_\ell!\ell^{m_\ell})}\bigotimes_{\ell=1}^r\mat K_{d,d^{2\ell-1}}^{\otimes m_\ell} \in \mathcal{M}_{d^{2r}\times d^{2r}}.$$
\end{enumerate}
\end{theorem}


Note that \citet[p.~168]{MN19} found the $r$th order differential of the log matrix determinant as $\diff^r\log|\bX|=(-1)^{r-1}(r-1)!\tr\{(\bX^{-1}\diff\bX)^r\}$. However, they did not give the $r$th order derivative since they did not have an identification theorem as our Theorem \ref{thm:ident}. Besides, in Theorem 8.1 in that same reference, the first order differential of the matrix determinant is given as $\diff |\bX| = |\bX| \tr(\bX^{-1} \diff \bX)=|\bX|\vec^\top\{(\bX^{-1})^\top\}\diff\vec\bX$, so that $\D|\bX|=|\bX|\vec\{(\bX^{-1})^\top\}$. This agrees with Theorem \ref{thm:matrix-det}(ii) since $\mat\Xi_{d,1}=\mat K_{d,d}$ and $\mat K_{d,d}\vec(\bX^{-1})=\vec\{(\bX^{-1})^\top\}$. But \citeauthor{MN19} provided no higher order differentials nor derivatives of $|\bX|$, which are found for an arbitrary order in our Theorem \ref{thm:matrix-det}.

\section{Connections to existing results} \label{sec:connect}

In Sections~\ref{sec:ident} and \ref{sec:chain}, we laid the foundations for a rigorous framework for a differential calculus for vector-valued functions of vector variables. In this section, we continue to elaborate it by contextualizing existing results within this framework.

\subsection{Taylor's theorem with vectorized derivatives} \label{sec:Taylor}

We return to our motivating example of Taylor approximations in the Introduction.
Whilst Taylor polynomials are well known for scalar functions of a vector variable, via their characterization with multi-indices in Equation~\eqref{eq:difr}, we observe that Theorem~\ref{thm:ident-vec} allows for their characterization with vectorized derivatives.
If $f\colon\mathbb R^d\to\mathbb R$ is a function such that every element in
$\D^{\otimes j}f(\bc), 0\leq j \leq r$ is piecewise continuous, then Theorem~3.11.10 from \citet{BL86} states that its $r$th order Taylor polynomial approximation is given by
$$f(\bc + \bu) = \sum_{j=0}^r \frac{1}{j!} (\bu^\top \D)^j f (\bc) + \re_{\bc}(\bm u),$$
where $\bu \in \mathbb{R}^d$ and $\re_{\bc}(\bm u)/\|\bm u\|^r\to0$ as $\bm u\to{\bm 0}$.
This form is not amenable for our purposes since it combines the infinitesimal $\bu$ with
the action of the differential operator $\D$.
Using the identity
$(\boldsymbol{a}^\top \boldsymbol{b})^j = (\boldsymbol{a}^\top \boldsymbol{b})^{\otimes j} = (\boldsymbol{a}^\top)^{\otimes j} \boldsymbol{b}^{\otimes j}$
for vectors $\boldsymbol{a}, \boldsymbol{b}$ of
the same length, we can extricate the role of $\D$ from $\bu$ to obtain an alternative expansion
\begin{equation}
\label{eq:Taylor}f(\bc + \bu) = \sum_{j=0}^r \frac{1}{j!} (\bu^\top)^{\otimes j}
\D^{\otimes j}f(\bc) + \re_{\bc}(\bm u).
\end{equation}
Equation~\eqref{eq:Taylor} is a stepping stone to the development of Taylor polynomials for a vector-valued function.

\begin{theorem}[Vector-valued Taylor approximation]
\label{thm:Taylor}
Let $\bc$ and $\bc + \bu$ be distinct points in an open subset $\Omega \subseteq \mathbb{R}^d$
such that the straight line segment joining $\bc$ and $\bc + \bu$ lies in $\Omega$.
Let $\bbf\colon\mathbb{R}^d \rightarrow \mathbb{R}^p$ be a vector-valued function
that is $r$ times continuously differentiable on $\Omega$.
The $r$th order Taylor polynomial approximation of $\bbf$ is given by
$$\bbf(\bc+\bu) = \sum_{j=0}^r\frac{1}{j!}\big\{ \bI_p \otimes (\bu^\top)^{\otimes j}\big\} \D^{\otimes j} \bbf(\bc) + \re_{\bc}(\bm u)$$
where $\re_{\bc}(\bm u)$ is such that $\re_{\bc}(\bm u)/\|\bm u\|^r\to0$ as $\bm u\to{\bm 0}$.
\end{theorem}

The proof of Theorem~\ref{thm:Taylor} was demonstrated in \citet[Theorem~5.8]{CD18}.

For the special case where the vector-valued function $\bbf$ can be expressed as the $s$th derivative of a scalar-valued function $g$, i.e. $g\colon \mathbb{R}^d \to \mathbb{R}$ and $\bbf\colon \mathbb{R}^d \to \mathbb{R}^{d^s}$ with $\bbf = \D^{\otimes s} g$, then we have
\begin{align*}
\bbf(\bc +\bu)
&= \sum_{j=0}^r \frac{1}{j!} [ \bI_{d^s} \otimes (\bu^\top)^{\otimes j}] \D^{\otimes s+j} g(\bc) + \re_{\bc}(\bm u) \\
&= \sum_{j=0}^r \frac{1}{j!} [ \bI_d^{\otimes s} \otimes (\bu^\top)^{\otimes j}] \S_{d,s+j} \D^{\otimes s+j} g(\bc) + \re_{\bc}(\bm u).
\end{align*}
The action of the symmetrizer matrix in the matrix product in the summand can therefore be also interpreted as permuting the order of each of the $s$ identity matrices $\bI_d$ and the $j$ infinitesimals $\bu$, rather than on $\D^{\otimes s+j} g(\bc)$. The latter are usually less easily expressed as an $(s+j)$-fold Kronecker product.

\subsection{General approximation of the identity with vectorized derivatives}\label{sec:mollifiers}

We analyze further the role of Taylor expansions in the approximation of the identity.  Let a kernel $K\colon\mathbb R^d\to\mathbb R$ be an integrable scalar-valued function with unit integral, and $f\colon\mathbb R^d\to\mathbb R$ be another integrable function. The convolution of $K$ and $f$ is defined by $K*f(\bx)=\int_{\mathbb R^d}K(\bx-\by)f(\by) \, \diff\by$, and it inherits the differentiability properties of $K$ \citep[][Theorem 9.3]{WZ77}. Therefore $K*f$ can be interpreted as a kernel-smoothed version of $f$. Moreover, if a rescaled version $K_h(\bx)=K(\bx/h)/h^d$ is considered, with a smoothing parameter $h>0$, then $K_h*f(\bx)\to f(\bx)$ in various senses as $h\to0$ \citep[][Section~9.2]{WZ77}.  When this convergence holds, then the family of functions $\{K_h\}_{h>0}$ is known as an approximation of the identity, or a mollifier. These convergence properties of $K_h*f$ rely on Taylor expansions where $h$ is an infinitesimal element.

The rescaling $K_h$ is commonly referred as a spherical rescaling, since it applies the same scaling factor for all coordinate directions $x_1, \dots, x_d$. An elliptical rescaling, $K_{\bm h}(\bx)=K(x_1/h_1,\dots,x_d/h_d)/(h_1\cdots h_d)$, where $\bm h=(h_1,\dots,h_d)$ is a vector of possibly different positive scaling factors, allows for a different rescaling for each coordinate direction. But the most general rescaling is obtained using $K_\bH(\bx)=|\bH|^{-1/2}K(\bH^{-1/2}\bx)$, where $\bH$ is a symmetric positive-definite matrix (i.e., $\bH>0$) and $\bH^{-1/2}$ is such that $\bH^{-1/2}\bH^{-1/2}=\bH^{-1}$. This general form, which subsumes the spherical and elliptical rescalings as special cases, additionally allows an arbitrary rotation before the elliptical rescaling, so the rescaling is no longer restricted to follow the coordinate directions. In the context where $f$ is a multivariate probability density function, \cite{WJ93} showed that approximations using this unconstrained scaling can lead to substantial gains in accuracy in statistical estimation.

The proof of the convergence of the spherical mollifiers $K_h*f$ carries over with minor adjustments to the general, unconstrained case of $K_\bH*f$, so that $K_\bH*f(\bx)\to f(\bx)$ as $\vec\bH\to0$ in the same senses as previously. However quantifying the rate of convergence of $K_\bH*f(\bx)- f(\bx)$ to zero requires a more involved analysis of the general approximation of the identity $\{K_\bH\}_{\bH>0}$.
We begin with
\begin{align*}
K_\bH*f(\bx)- f(\bx)&=\int_{\mathbb R^d}\{f(\bx-\by)-f(\bx)\}K_\bH(\by) \, \diff\by\\
&=\int_{\mathbb R^d}\big\{f(\bx-\bH^{1/2}\bz)-f(\bx)\big\}K(\bz) \, \diff\bz.
\end{align*}
If $f$ is $k$ times differentiable at $\bx$, then we can use a Taylor expansion with vectorized derivatives to approximate $f(\bx-\bH^{1/2}\bz)-f(\bx)\simeq\sum_{j=1}^{k} (-1)^j\frac1{j!}\D^{\otimes j}f(\bx)^\top(\bH^{1/2})^{\otimes j}\bz^{\otimes j}$. Furthermore, if $K$ is a kernel of order $k$, i.e. meaning that $\bmu_j(K)=\int_{\mathbb R^d}\bz^{\otimes j}K(\bz)\diff\bz=\mathbf 0$ for $j=1,\dots,k-1$ and $\bmu_{k}(K)=\int_{\mathbb R^d}\bz^{\otimes k}K(\bz)\,\diff\bz\neq\mathbf0$, then it follows that
\begin{equation}\label{eq:KHf}
K_\bH*f(\bx)- f(\bx)\simeq\frac1{k!}\D^{\otimes k}f(\bx)^\top(\bH^{1/2})^{\otimes k}\bmu_{k}(K).
\end{equation}
Equation (\ref{eq:KHf}) shows that vectorized derivatives in the Taylor polynomial allow for the separation of a matrix-valued infinitesimal $\bH^{1/2}$ and the free variable $\bx$, so that $K_\bH*f(\bx)-f(\bx)$ can be separated into a vectorized derivative of $f$, an infinitesimal element, and a vectorized moment of $K$.
So it is straightforward to assert, for a general approximation of the identity $\{K_\bH\}_{\bH>0}$ with a $k$th-order kernel, that $K_\bH*f(\bx)- f(\bx)$ converges to zero at the same rate as $(\bH^{1/2})^{\otimes k}$.

\subsection{Vector Hermite polynomials} \label{sec:hermite-faa}

Higher order differential analysis is of intense interest for the special case of the infinitely differentiable Gaussian density functions, due to its numerous statistical applications \citep[see][]{CD15}.
Let $\phi_\bSigma$ be the Gaussian density with mean $\bm 0$ and variance $\bSigma$, i.e. $\phi_\bSigma(\bx) = (2\pi)^{-1/2} |\bSigma|^{-1/2} \exp(-\bx^\top \bSigma^{-1} \bx/2)$. Then, Equation~(2.1) in the pioneering paper of \citet{Hol96a} states that
\begin{equation*}
	\D^{\otimes r} \phi_\bSigma(\bx) = (-1)^r (\bSigma^{-1})^{\otimes r} \boldsymbol{\mathcal{H}}_r (\bx; \bSigma) \phi_\bSigma(\bx),
\end{equation*}
where $\boldsymbol{\mathcal{H}}_r$ is the $r$th order vector Hermite polynomial, defined
in Equation~(3.3) in the same paper, as
\begin{equation}\label{eq:Hermite}
\boldsymbol{\mathcal{H}}_r (\bx; \bSigma) = r! \sum_{j=0}^{\lfloor r/2\rfloor} \frac{(-1)^j}{j! (r-2j)!2^j } \S_{d,r} \{\bx^{\otimes r-2j} \otimes
(\vec \bSigma)^{\otimes j}\}.
\end{equation}

Here we note that it is possible to derive a simple proof for the $r$th derivative of the Gaussian density from our Fa\`a di Bruno's formula.  Let $g(y)=(2\pi)^{-d/2} |\bSigma|^{-1/2}\exp(y)$ and $f(\bx)=-\bx^\top\bSigma^{-1}\bx/2$, so that
$\phi_\bSigma(\bx)=(g\circ f) (\bx)$. Then, $\D^{\otimes r}g(y)=g(y)$ for all $r$, and therefore $\D^{\otimes r}g\{f(\bx)\}=\phi_\bSigma(\bx)$. On the other hand, the differentials and derivatives of $f$ are $\diff f(\bx) = -\frac12(\diff \bx^\top \bSigma^{-1} \diff\bx + \bx^\top \bSigma^{-1} \diff\bx) = - \bx^\top \bSigma^{-1} \diff\bx$ so that $\D f(\bx) = -\bSigma^{-1} \bx$, and $\diff^2 f(\bx)=-(\diff \bx^\top)\bSigma^{-1}\diff\bx=-(\vec\bSigma^{-1})^{\top}(\diff\bx)^{\otimes2}$ so that $\D^{\otimes 2}f(\bx)=-\S_{d,2}\vec\bSigma^{-1}=-\vec\bSigma^{-1}$ from Theorem \ref{thm:ident} (because $\bSigma^{-1}$ is symmetric), and $\D^{\otimes r}f(\bx)=\bm 0$ for all $r\geq3$.

Then Theorem~\ref{thm:deriv-comp} asserts that
\begin{equation}\label{eq:DrphiSigma}
\D^{\otimes r} \phi_\bSigma(\bx)= \sum_{\bbm\in\mathcal J_{r}}\pi_{\bbm} \D^{\otimes |\bbm|}g\{f(\bx)\} \S_{d,r} \bigotimes_{\ell=1}^r\{\D^{\otimes \ell}f(\bx)\}^{\otimes m_\ell}
\end{equation}
and recall that $\bbm = (m_1, \dots, m_r)\in\mathcal J_r$ are the non-negative solutions to the linear Diophantine equation $1 \cdot m_1 +  2 \cdot m_2 + \dots  + r \cdot m_r = r$. Since $\D^{\otimes r}f(\bx) = \bm 0$ for all $r\geq 3$, then the terms in Equation~\eqref{eq:DrphiSigma} will be identically zero whenever $m_\ell>0$ for some $\ell\geq 3$. So it suffices to consider $\bbm\in\mathcal J_r$ with  $m_3 = \dots = m_r =0$, which simplifies the Diophantine equation to $m_1 + 2 m_2 = r$. Since $m_1\geq0$ and $m_2\in\mathbb N_0$, the former equation implies that $m_2\leq\lfloor r/2\rfloor$, so all its solutions are given by $m_1=r-2j$, $m_2=j$ for $j=0, 1, \dots, \lfloor r/2\rfloor$.
Then the coefficient $\pi_{\bbm}=r!/\prod_{\ell=1}^r\{m_\ell!(\ell!)^{m_\ell}\}$ has the simpler form $\pi_\bbm = r!/\{(r-2j)! j! 2^{j}\}.$
Combining these with
$\D^{\otimes r}g\{f(\bx)\}=\phi_\bSigma(\bx)$ for all $r$, $\D f(\bx) = -\bSigma^{-1} \bx$, and $\D^{\otimes 2}f(\bx)=-\vec\bSigma^{-1}$, we have
\begin{align*}
\D^{\otimes r} \phi_\bSigma(\bx)
&= \phi_\bSigma(\bx) \sum_{j=0}^{\lfloor r/2 \rfloor} \frac{r!}{(r-2j)! j! 2^{j}} \S_{d,r} \big[ \{\D f(\bx)\}^{\otimes r-2j} \otimes \{\D^{\otimes 2}f(\bx)\}^{\otimes j} \big] \\
&= \phi_\bSigma(\bx) \sum_{j=0}^{\lfloor r/2 \rfloor} \frac{r!}{(r-2j)! j! 2^{j}} \S_{d,r} \big\{ (-\bSigma^{-1}\bx)^{\otimes r-2j} \otimes (-\vec \bSigma^{-1})^{\otimes j} \big\} \\
&= (-1)^r (\bSigma^{-1})^{\otimes r} \phi_\bSigma(\bx)  \sum_{j=0}^{\lfloor r/2 \rfloor} \frac{ (-1)^j r!}{(r-2j)! j! 2^{j}} \S_{d,r} \big\{ \bx^{\otimes r-2j} \otimes (\vec \bSigma)^{\otimes j} \big\}.
\end{align*}
The summation is identical to the vector Hermite polynomial $\bm{\mathcal{H}}_r(\bx;\bSigma)$ introduced previously in (\ref{eq:Hermite}). This derivation is an alternative to the one based on a formal Taylor series expansion provided by \citet{Hol96a}.

\subsection{Individual partial derivatives within vectorized derivatives} \label{sec:local}

Our proposed derivative consists of a systematic ordering of all the possible higher order partial derivatives as a single vectorized derivative. This is a basic property in building our proposed algebra of differentials.
Nonetheless, there are situations where explicit knowledge of the location of certain mixed partial derivatives $\D^r_{i_1\cdots i_r}f$ is important, e.g., (i) to diagonalize the derivative which involves the extraction of elements on the main diagonal $\partial^r/\partial x_i^r, i=1, \dots, d$; (ii) to obtain the Laplacian $\triangle=\sum_{i=1}^d\partial^2/\partial x_i^2$; (iii) to express a multivariate density function $f$ in terms of its distribution function $F$ as $f=\partial^d F/(\partial x_1\cdots\partial x_d)$. Whilst this is trivial for multi-index or matrix or tensor representations of higher order derivatives, for vectorized representations it requires a separate procedure.

Given the indices $i_1,\dots,i_r\in\{1,\dots,d\}$, the problem is to locate the position $p=p(i_1,\dots,i_r)\in\{1,\dots,d^r\}$ where the partial derivative $\D^r_{i_1\cdots i_r}$ lies within the vector $\D^{\otimes r}$, i.e., such that $\D^r_{i_1\cdots i_r}=(\D^{\otimes r})_p$. But starting from $\be_i^\top\D=\partial/\partial x_i$ it is clear that $\D^r_{i_1\dots i_r}=(\be_{i_1}\otimes\cdots\otimes\be_{i_r})^\top\D^{\otimes r}$, so the problem reduces to locating the only nonzero element of $\bigotimes_{\ell=1}^r\be_{i_\ell}$ and, reasoning as in Lemma~1.3.1 in \citet{KvR05}, it follows that such an element is at position $p=p(i_1,\dots,i_r)=1+\sum_{j=1}^r(i_j-1)d^{r-j}$. Hence, locating an individual partial derivative within the derivative vector is trivial.

Moreover, since our regularity conditions ensure that $\D^r_{i_1\cdots i_r}=\D^r_{i_{\sigma(1)}\cdots i_{\sigma(r)}}$ for any permutation $\sigma\in\mathcal P_r$, then in fact any map $p_\sigma(i_1,\dots,i_r)=p(i_{\sigma(1)},\dots,i_{\sigma(r)})=1+\sum_{j=1}^r(i_{\sigma(j)}-1)d^{r-j}$ is also valid to locate $\D^r_{i_1\cdots i_r}$ within $\D^{\otimes r}$. For instance, \cite{CD15} used $p_\tau$ for the permutation $\tau$ such that $\tau^{-1}(j)=r-j+1$.

The map $p\colon\{1,\dots,d\}^r\to\{1,\dots, d^r\}$, defined between these two sets of the same cardinality, can be shown to be a bijection. Hence, its inverse function $p^{-1}$ is also useful to find out the multi-index form of a partial derivative located at a given coordinate of $\D^{\otimes r}$, as detailed in Appendix 2 in \cite{CD15}. Intuitively, $p_\tau^{-1}$ is in effect a change of base of an integer in $\{1,\dots,d^r\}$ from  base-10 to base-$d$, though using the numerals drawn from $\{1,\dots,d\}$ instead of the usual $\{0,\dots,d-1\}$.

\subsection{Vectorized higher order moments and cumulants} \label{sec:char}

Let $\bX$ be a $d$-variate random vector. For a multi-index $i_1,\dots,i_r\in\{1,\dots,d\}$, it is common to refer to the expected value $\mathbb E(X_{i_1}\cdots X_{i_r})$ as a mixed moment of order $r$. There are many of these individual real-valued mixed moments, and they are all contained in the $r$th order vectorized moment $\bmu_r=\mathbb E(\bX^{\otimes r})\in\mathbb R^{d^r}$. As is the case for the individual mixed partial derivatives, even if particular mixed moments can be of interest in some situations, the whole vectorized moment is needed for the expansion of the characteristic function or the moment-generating function of $\bX$.

Indeed, the moments of a random variable are closely related to the employed notion of derivative since the former can be obtained via the derivative of the moment-generating function at zero. Hence, the arrangement of the $r$th order moment is immediately inherited from the corresponding layout of the $r$th order derivative. In our case, if we compute the expectations and derivatives from first principles, then we demonstrate in Lemma \ref{lem:mom} in Section \ref{app:moments} that
\begin{equation} \label{eq:char}
\bmu_r=\E (\bX^{\otimes r}) = \D^{\otimes r} M_\bX (\bt) |_{\bt=\bm 0}
\end{equation}
where $M_\bX (\bt) = \E \{ \exp(\bt^\top \bX)\}$ denotes the moment-generating function of $\bX$.
This approach via vectorized derivatives was the only successful tool to find the moments of arbitrary order of a multivariate normal vector \citep{Hol88}, after several authors previously focused only on finding moments of certain particular orders.

\citet[p.~173]{KvR05} entertained the possibility of defining the $r$th moment as $\mathbb E(\bX^{\otimes r})$, though they eventually argued against this configuration due to that ``it is complicated to show where the moments of interest are situated in the vector''. With the map $p$ introduced in Section \ref{sec:local}, to carry out this localization no longer involves any difficult procedures. Moreover, these authors define a derivative that leads to the $r$th moment of $\bX$ as $\mathbb E\{\bX(\bX^\top)^{\otimes r-1}\}\in\mathcal M_{d\times d^{r-1}}$ \citep[Theorem 2.1.1]{KvR05}. This indeed exhibits the nice feature that the covariance is a $d\times d$ matrix (analogously to the Hessian matrix), but it does not represent a conceptual advantage over $\bmu_r=\mathbb E(\bX^{\otimes r}) = \vec (\E\{\bX(\bX^\top)^{\otimes r-1}\})$ in terms of higher order moments or the location of individual mixed moments, so we maintain our preference for the purely vectorized form $\bmu_r$.

The cumulants of the random vector $\bX$ provide an alternative to moments, which is particularly useful for Edgeworth and related expansions of distributions \citep{JTT21b}. The cumulant-generating function of $\bX$ is given by $C_\bX(\bt)=\log M_\bX(\bt)$ for $\bt\in\mathbb R^d$ and, by analogy with the moments, the $r$th order vectorized cumulant is defined as ${\bm \kappa}_r=\D^{\otimes r} C_\bX(\bt)|_{\bt=\bm 0}$. Recently, \cite{JTT21} explored in detail the special cases of $\bm\kappa_3$ and $\bm\kappa_4$ to define multivariate analogues of the skewness and kurtosis of multivariate distributions. Cumulants and moments are closely connected and it is useful to express the former in terms of the latter and vice versa \citep[see][]{Hol85b}. These relationships are usually difficult to describe, however, they follow easily from the higher order chain rule in Theorem \ref{thm:deriv-comp}. By expressing $M_\bX(\bt)=(g\circ C_\bX)(\bt)$ with $g(y)=\exp(y)$, it readily follows from Theorem \ref{thm:deriv-comp} that
$$\bmu_r = \sum_{\bbm\in\mathcal J_{r}}\pi_{\bbm}\S_{d,r}\bigotimes_{\ell=1}^r{\bm\kappa}_\ell^{\otimes m_\ell},$$
which corresponds to Theorem 4.1(i) in \cite{Hol85b}. Reciprocally, by writing $C_\bX(\bt)=(g\circ M_\bX)(\bt)$ with $g(y)=\log(y)$, Theorem \ref{thm:deriv-comp} immediately gives
$${\bm\kappa}_r = \sum_{\bbm\in\mathcal J_{r}}\pi_{\bbm}(-1)^{|\bbm|-1}(|\bbm|-1)!\S_{d,r}\bigotimes_{\ell=1}^r\bmu_\ell^{\otimes m_\ell},$$
which agrees with Theorem 4.1(iii) in \cite{Hol85b}.

\subsection{Unique vectorized moments and partial derivatives}

If $f\colon\mathbb R^d\to\mathbb R$ is $r$ times differentiable then $\D^{\otimes r}f$ is a vector of length $d^r$, but not all of its entries are distinct. The same occurs for the vectorized moment $\bmu_r=\mathbb E(\bX^{\otimes r})$, which includes many redundant mixed moments. This is not a problem for theoretical developments; on the contrary, having a neat configuration has proved to be beneficial (and even essential) for building an algebra of differentials and unveil general results that are valid for any arbitrary order $r$.

However, for computational purposes, it can be more efficient to first calculate only the distinct elements and, if required, then redistribute them to form the full derivative vector of length $d^r$. The number of distinct partial derivatives is the same as the number of distinct mixed moments, and they equal the number of monomials of degree $r$ in $d$ variables, i.e. $\binom{d+r-1}{r}$; see \citet[][Section II.5]{F68}. As $d$ increases, the proportion of distinct elements in either $\D^{\otimes r}f$ or $\bmu_r$ approaches $1/r!$, so the time savings from computing only the distinct elements of these vectors can be considerable, even for moderate values of $r$.

The vector containing all the distinct $r$th order mixed moments is called the $r$th order minimal moment in \citet[][Section 2.1.6]{KvR05}, and Theorem~2.1.10 in this reference shows how to obtain $\bmu_r$ from this minimal representation. The distinct partial derivatives of a function are called the unique partial derivatives in \citet[Section~5]{CD15}, where an efficient recursive algorithm to compute these unique partial derivatives (and the subsequent entire derivative vector) of the multivariate Gaussian density is exhibited.

\section{Conclusion} \label{sec:conc}

We have introduced a rigorous, comprehensive framework for the differential analysis for vector-valued functions of vector variables. The foundations of this analytic framework are the existence and uniqueness of the identifications between the differentials and the derivatives of any order. These existence and uniqueness properties have hitherto resisted a sufficiently complete characterization. The latter in turn facilitates the construction of an algebra of differentials/derivatives that is an intuitive generalization of that which exists for scalar-valued functions of scalar/vector variables.

We established two fundamental rules of this algebra in order to  compute higher order derivatives: (i)~a  Leibniz rule for the product of two functions  and (ii) Fa\`a di Bruno's rule for the composition of two functions. In addition to these foundational results, we {established explicit derivatives for the vector/matrix monomial, matrix trace, matrix inverse and matrix determinant;} and that well-known analytic results (such as Taylor's theorem, Hermite polynomials, and the relationship between moments and cumulants) can be re-cast within this framework, often with a considerable simplification of their development, and crucially with the ability to be generalized in an intuitive manner to any dimension and to any derivative order.

\paragraph{Acknowledgements.} The authors are grateful to Professor Bj\"{o}rn Holmquist for kindly sharing a copy of his unpublished research report. J. E. Chac\'on has been partially supported by the Spanish Ministerio de Ciencia e Innovaci\'on grant PID2019-109387GB-I00 and the Junta de Extremadura grant GR18016.

\appendix

\section{Appendix: Supporting lemmas and proofs}

\subsection{Scalar-valued identification} \label{app:ident}
\begin{lemma} \label{lem:ident-zero}
Let $\ba\in\mathbb R^d$. Then, $\ba^\top\bx^{\otimes r}=0$ for all $\bx\in\mathbb R^d$ if and only if $\S_{d,r}\ba=0$.
\end{lemma}

\begin{proof}
If $\S_{d,r}\ba=0$, using the properties of the symmetrizer matrix \citep{Sch03}, then
$$\ba^\top\bx^{\otimes r}=\ba^\top\S_{d,r}\bx^{\otimes r}=\ba^\top\S_{d,r}^\top\bx^{\otimes r}=(\S_{d,r}\ba)^\top\bx^{\otimes r}=0.$$
To show the reverse implication assume that $\ba^\top\bx^{\otimes r}=0$ for all $\bx\in\mathbb R^d$ and note, from the explicit representation of $\S_{d,r}$, that it suffices to show that for any choice of $i_1,\dots,i_r\in\{1,\dots,d\}$, that we have
\begin{equation}\label{aPr}
\ba^\top\sum_{\sigma\in\mathcal P_r}\big(\be_{i_{\sigma(1)}}\otimes\cdots\otimes\be_{i_{\sigma(r)}}\big)=0.
\end{equation}
In order to prove Equation~\eqref{aPr}, we introduce some notation. For a vector $\bj=(j_1,\dots,j_k)$ of indices, let $|\{\bj\}|$ denote the number of its distinct coordinates; that is, the cardinality of the set $\{j_1,\dots,j_k\}$. Given $k$ different indices $i_1,\dots,i_k\in\{1,\dots,d\}$ with $k\leq r$, let $\mathcal I_{r,p}=\mathcal I_{r,p}(i_1,\dots,i_k)$ be the set of $r$-dimensional vectors of indices in $\{i_1,\dots,i_k\}$ having exactly $p$ different coordinates ($p\leq k$); that is,
$$\mathcal I_{r,p}=\mathcal I_{r,p}(i_1,\dots,i_k)=\big\{\bj=(j_1,\dots,j_r)\in\{i_1,\dots,i_k\}^r\colon |\{\bj\}|=p\big\}.$$

First we claim that, from the fact that $\ba^\top\bx=0$ for all $\bx\in\mathbb R^d$, it follows that
\begin{equation}\label{aIrk}
\ba^\top\sum_{\bj\in\mathcal I_{r,k}}(\be_{j_1}\otimes\cdots\otimes\be_{j_r})=0.
\end{equation}
To assert Equation~\eqref{aIrk}, we proceed by induction on $k$. For $k=1$, this statement affirms that for any $i\in\{1,\dots,d\}$ we have $\ba^\top\be_i^{\otimes r}=0$, which is trivially true by taking $\bx=\be_i$ in the hypothesis. So assume by induction that the result is true for any set of $k-1$ different indices and we will demonstrate that Equation~\eqref{aIrk} holds for any $k\leq r$ different indices $i_1,\dots,i_k\in\{1,\dots,d\}$. Notice that
$\{i_1,\dots,i_k\}^r=\bigcup_{p=1}^k\mathcal I_{r,p}$ with $\mathcal I_{r,p}\cap\mathcal I_{r,q}=\varnothing$ for $p\neq q$, so that taking $\bx=\be_{i_1}+\cdots+\be_{i_k}$ it follows that
$$0=\ba^\top(\be_{i_1}+\cdots+\be_{i_k})^{\otimes r}=\ba^\top\sum_{\bj\in\{i_1,\dots,i_k\}^r}(\be_{j_1}\otimes\cdots\otimes\be_{j_r})=\ba^\top\sum_{p=1}^k\sum_{\bj\in\mathcal I_{r,p}}(\be_{j_1}\otimes\cdots\otimes\be_{j_r}).
$$
By the induction hypothesis the right-hand-side reduces to $\ba^\top\sum_{\bj\in\mathcal I_{r,k}}(\be_{j_1}\otimes\cdots\otimes\be_{j_r})$, so this yields Equation~\eqref{aIrk}.

Finally, we use Equation~\eqref{aIrk} to assert Equation~\eqref{aPr}. Consider any indices $i_1,\dots,i_r\in\{1,\dots,d\}$ and denote $|\{i_1,\dots,i_r\}|=k$, with $k\leq r$. If all the indices are different, then $k=r$ and, moreover,
$$\{(i_{\sigma(1)},\dots,i_{\sigma(r)})\colon \sigma\in\mathcal P_r\}=\mathcal I_{r,r}(i_1,\dots,i_r),$$
so in this case Equation~\eqref{aPr} follows directly from Equation~\eqref{aIrk}. When $k<r$, it is sufficient that to demonstrate that the sum on the left hand side of Equation~\eqref{aPr} is proportional to the sum on the left hand side of Equation~\eqref{aIrk}. More precisely,  denote $\{i_1,\dots,i_r\}=\{\imath_1,\dots,\imath_k\}$, with $\imath_1,\dots,\imath_k\in\{1,\dots,d\}$, to represent the distinct coordinates of $(i_1,\dots,i_r)$. For any $\ell\in\{1,\dots,k\}$, write $r_\ell$ for the number of times that $\imath_\ell$ appears in $(i_1,\dots,i_r)$, so that $1\leq r_\ell\leq r$ for all $\ell=1,\dots,k$ and $r_1+\dots+r_k=r$. Then,
$$\sum_{\sigma\in\mathcal P_r}(\be_{i_{\sigma(1)}}\otimes\dots\otimes \be_{i_{\sigma(r)}})=r_1!\cdots r_k!\sum_{\bj\in\mathcal I_{r,k}(\imath_1,\dots,\imath_k)}(\be_{j_1}\otimes\cdots\otimes\be_{j_r}).$$
The former equation can be explicitly shown in the same way as in the combinatorial proof that $r!/(r_1!\cdots r_k!)$ is the number of permutations with repetition of the elements of $\{\imath_1,\dots,\imath_k\}$ with the element $\imath_\ell$ repeated $r_\ell$ times, $\ell=1,\dots,k$ \citep[see][Theorem 2.4.2]{Bru10}. The set of all permutations of $(i_1,\dots,i_r)$ is, in fact, a multi-set (there are repeated elements) of cardinality $r!$, whose elements are all the aforementioned permutations with repetitions, so that each of these permutations with repetitions appears exactly $r_1!\cdots r_k!$ times in the set of all permutations of $(i_1,\dots,i_r)$. This establishes the proportionality of the summations on the left hand sides of Equations~\eqref{aPr} and \eqref{aIrk}.
\end{proof}

\subsection{Vector-valued identification} \label{app:ident-vec}

\begin{proof}[Proof of Lemma~\ref{lem:invvec}]
This result elaborates on the entry `What is the inverse of the $\vec$ operator?' of the webpage {\tt math.stackexchange.com} which contains a slightly incomplete (and different) proof.

We begin by showing that $\vec_{m,n}^{-1}(\vec\bA)=\bA$ for any $\bA\in\mathcal M_{m\times n}$. Let $\be_i\in\mathbb R^{n}$ be the $i$th column of $\bI_{n}$. Then, $\ba_i=\bA\be_i$ is the $i$th column of $\bA$ and we can write $\bA=\sum_{i=1}^n\ba_i\be_i^\top$ and $\bI_n=\sum_{i=1}^n\be_i\be_i^\top$. By making use of the usual properties of the vec operator and the Kronecker product we have
\begin{align*}
\vec_{m,n}^{-1}(\vec\bA)&=\{(\vec^\top\bI_n)\otimes\bI_m\}(\bI_n\otimes \vec\bA)
 =\sum_{i=1}^n(\be_i^\top\otimes\be_i^\top\otimes \bI_m)(\bI_n\otimes\vec\bA)\\
&=\sum_{i=1}^n(\be_i^\top\bI_n)\otimes\{(\be_i^\top\otimes \bI_m)\vec\bA\}
 =\sum_{i=1}^n\be_i^\top\otimes(\bA\be_i)=\sum_{i=1}^n\ba_i\be_i^\top=\bA.
\end{align*}

On the other hand, we need to show that $\vec\{\vec_{m,n}^{-1}(\ba)\}=\ba$ for any $\ba\in\mathbb R^{mn}$. Using Theorem 3.10 in \cite{MN19},
$$\vec(\bI_n\otimes\ba)=[\{(\bI_n\otimes\mat K_{1,n})\vec\bI_n\}\otimes \bI_{mn}]\ba=\{(\vec\bI_n)\otimes\bI_{mn}\}\ba,$$
since the commutation matrix satisfies $\mat K_{1,n}=\bI_n$. Therefore,
\begin{align*}
\vec\{\vec_{m,n}^{-1}(\ba)\}&=\vec[\{(\vec^\top\bI_n)\otimes\bI_m\}(\bI_n\otimes \ba)]
 =\{\bI_n\otimes(\vec^\top\bI_n)\otimes\bI_n\}\vec(\bI_n\otimes \ba)\\
&=\{\bI_n\otimes(\vec^\top\bI_n)\otimes\bI_n\}\{(\vec\bI_n)\otimes\bI_{mn}\}\ba\\
&=([(\bI_n\otimes\vec^\top\bI_n)\{(\vec\bI_n)\otimes\bI_n\}]\otimes\bI_m)\ba.
\end{align*}
To finish the proof, it suffices to establish the identity
\begin{equation}\label{eq:Ivec}
(\bI_n\otimes\vec^\top\bI_n)\{(\vec\bI_n)\otimes\bI_n\}=\bI_n.
\end{equation}
This can be shown by writing $\bI_n=\sum_{i=1}^n\be_i\be_i^\top$ as above, so that
\begin{align*}
(\bI_n\otimes\vec^\top\bI_n)\{(\vec\bI_n)\otimes\bI_n\}&=\sum_{i,j=1}^n(\bI_n\otimes\be_i^\top\otimes\be_i^\top)(\be_j\otimes\be_j\otimes\bI_n)\\
&=\sum_{i,j=1}^n\be_j\otimes(\be_i^\top\be_j)\otimes\be_i^\top=\sum_{i=1}^n\be_i\otimes\be_i^\top=\bI_n,
\end{align*}
thus yielding Equation~\eqref{eq:Ivec}.
\end{proof}

\begin{proof}[Proof of Theorem \ref{thm:ident-vec}]
(i) The text preceding the theorem statement establishes this.

(ii) Let $\ba_i$ be the $i$th column of $\bA \in \mathcal{M}_{d^r \times p}$, $i=1,\dots, p$. Since
$\D^{\otimes r}f_i(\bc)^\top\bu^{\otimes r}=\diff^rf_i(\bc;\bu)=\ba_i^\top\bu^{\otimes r}$ for all $\bu\in\mathbb R^d$, then, by a component-wise application of Theorem~\ref{thm:ident}(ii), we have $\D^{\otimes r}f_i(\bc)=\S_{d,r} \ba_i$. That is,
$$\D^{\otimes r} \bbf(\bc)=\begin{bbmatrix} \S_{d,r}\ba_1\\ \hline \vdots \\ \hline \S_{d,r}\ba_p\end{bbmatrix}=(\bI_p\otimes\S_{d,r})\begin{bbmatrix} \ba_1\\ \hline \vdots \\ \hline \ba_p\end{bbmatrix}=(\bI_p\otimes\S_{d,r})\vec\bA.
$$
The second part of (ii) follows immediately if we set $\bA=\vec_{d^r,p}^{-1}\ba$ since $\ba = \vec \bA$.
\end{proof}

\subsection{Iterative identification} \label{app:ident-iter}

\begin{proof}[Proof of Theorem~\ref{thm:ident-iter}]
First, consider a real-valued function $f\colon\mathbb{R}^d \to \mathbb{R}$. Denote $\bm g = \D^{\otimes (r-1)} f$. Since $\bm g\colon\mathbb R^d\to\mathbb R^{d^{r-1}}$ is a vector-valued function with $\D\bg=\D^{\otimes r}f$, applying Theorem~\ref{thm:ident-vec}(ii) to the first differential of $\bm g$ we have that, if $\bB \in \mathcal{M}_{d \times d^{r-1}}$ satisfies $\diff \bm g(\bc;\bu) = \bB^\top \bu$ for all $\bu \in \mathbb{R}^d$, then it must be $\D^{\otimes r}f(\bc)=\D \bm g(\bc) = \vec \bB$.

For a vector-valued $\bbf\colon\mathbb{R}^d \to \mathbb{R}^p$, with components $\bbf = (f_1, \dots, f_p)$, suppose that $\bB\in\mathcal M_{d \times pd^{r-1}}$ satisfies $\diff \{\D^{\otimes (r-1)} \bbf(\bc;\bu)\} =\mat B^\top \bu$ for all $\bu\in\mathbb R^d$. If we write $\bB^\top$ as $p$ stacked block matrices $\bB_1^\top,\dots,\bB_p^\top \in\mathcal M_{d^{r-1}\times d}$, then the previous assumption entails that
$$
 \begin{bbmatrix} \diff \{\D^{\otimes (r-1)} f_1\}(\bc;\bu) \\ \hline \vdots \\ \hline \diff \{\D^{\otimes (r-1)} f_p\}(\bc;\bu)\end{bbmatrix} =\diff \{\D^{\otimes (r-1)} \bbf\}(\bc;\bu) =
\bB^\top \bu =  \begin{bbmatrix} \bB_1^\top \\ \hline \vdots \\ \hline \bB_p^\top \end{bbmatrix} \bu
= \begin{bbmatrix} \bB_1^\top \bu \\ \hline \vdots \\ \hline \bB_p^\top \bu \end{bbmatrix}
$$
so that $\diff \{\D^{\otimes (r-1)} f_i\}(\bc;\bu)= \bB_i^\top \bu$ for all $\bu \in \mathbb{R}^d$, for $i=1,\dots,p$. By the above argument for scalar-valued functions,
this implies that $\D^{\otimes r}f_i(\bc)= \vec \bB_i$ for $i=1,\dots,p$; that is,
\begin{align*}
\D^{\otimes r} \bbf(\bc) = \begin{bbmatrix} \D^{\otimes r} f_1(\bc) \\ \hline \vdots \\ \hline \D^{\otimes r} f_p(\bc) \end{bbmatrix}
= \begin{bbmatrix} \vec \bB_1 \\ \hline \vdots \\ \hline \vec \bB_p \end{bbmatrix}
&= \vec  \left[\bgroup\def\arraystretch{0.8}\begin{array}{@{}c|c|c@{}}  \bB_1  & \dots &  \bB_p \end{array}\egroup\right]  =\vec \bB.
\qedhere
\end{align*}
\end{proof}

\subsection{Product and chain rules} \label{app:prod-chain}

\begin{proof}[Proof of Theorem~\ref{thm:deriv-const}]
(i) Since $\diff^rf(\bc;\bu)=\mathsf D^{\otimes r}f(\bc)^\top\bu^{\otimes r}$ for all $\bu\in\mathbb R^d$, it immediately follows that
$\diff^r(\ba f)(\bc;\bu)= \ba \diff^rf(\bc;\bu)=\ba \mathsf D^{\otimes r}f(\bc)^\top\bu^{\otimes r}$ for all $\bu\in\mathbb R^d$. So from Theorem~\ref{thm:ident-vec} its derivative is $ \D^{\otimes r}(\ba f)(\bc) = \vec \{\S_{d,r} \mathsf D^{\otimes r}f(\bc) \ba^\top\} = \vec \{\mathsf D^{\otimes r}f(\bc) \ba^\top\}=\ba\otimes\mathsf D^{\otimes r}f(\bc)$.

(ii) Let us write $\bB =\vec_{d^r,p}^{-1}\D^{\otimes r}f(\bc)\in \mathcal{M}_{d^r \times p}$ so that $\diff^r \bbf(\bc;\bu)=\bB^\top\bu^{\otimes r}$ for all $\bu\in\mathbb R^d$, which by Theorem~\ref{thm:ident-vec} implies $\mathsf D^{\otimes r} \bbf(\bc) = \vec (\S_{d,r} \bB)$.
For a vector $\ba\in\mathbb R^q$, then it is easy to check that $\diff^r (\ba \otimes \bbf)(\bc;\bu)=\ba \otimes \diff^r \bbf(\bc;\bu) = \ba \otimes (\bB^\top \bu^{\otimes r}) = (\ba \otimes \bB^\top) \bu^{\otimes r}$ for all $\bu\in\mathbb R^d$. Therefore, $\D^{\otimes r} (\ba \otimes \bbf)(\bc) = \vec \{\S_{d,r} (\ba^\top \otimes \bB)\} = \vec \{\ba^\top \otimes (\S_{d,r}\bB)\} = \ba \otimes \vec (\S_{d,r}\bB)= \ba \otimes \D^{\otimes r}  \bbf(\bc)$.

For a matrix $\bA\in\mathcal M_{q\times p}$, then it can be shown that $\diff^r( \bA \bbf)(\bc;\bu)=\bA\diff^r\bbf(\bc;\bu)  = \bA \bB^\top\bu^{\otimes r}$ for all $\bu\in\mathbb R^d$. Therefore, $\D^{\otimes r} (\bA \bbf)(\bc) = \vec (\S_{d,r} \bB\bA^\top) = (\bA \otimes \bI_{d^r}) \vec (\S_{d,r} \bB) = (\bA \otimes \bI_{d^r})  \D^{\otimes r}  \bbf(\bc)$.
\end{proof}

\begin{proof}[Proof of Theorem~\ref{thm:deriv-prod}]
(i) Note that for $p=q=1$ the Kronecker product coincides with the usual product. Hence, part (i) immediately follows from part (ii).

(ii) The first goal is to show that the differential of the Kronecker product satisfies
\begin{equation}\label{eq:drfg}
\diff^r (\bbf \otimes \bg)(\bc;\bu) = \sum_{j=0}^r\binom{r}{j} \diff^{r-j} \bbf(\bc;\bu) \otimes \diff^j \bg(\bc;\bu),
\end{equation}
where it is understood that $\diff^0\bbf(\bc;\bu)=\bbf(\bc)$ and $\diff^0\bg(\bc;\bu)=\bg(\bc)$. The proof of this fact follows closely the commonly exhibited reasoning for binomial expansions: this is true for $r=1$ since the usual Kronecker product rule for the first differential yields  $\diff (\bbf \otimes \bg)(\bc;\bu) =  \diff \bbf(\bc;\bu) \otimes \bg (\bc) + \bbf(\bc) \otimes \diff \bg(\bc;\bu)$ \citep[][p. 164]{MN19}. Then, Equation (\ref{eq:drfg}) follows by induction on $r$.

From Equation~\eqref{eq:drfg} and Theorem \ref{thm:ident-vec}(i) we obtain
\begin{align*}
\diff^r (\bbf \otimes \bg)(\bc;\bu)
&= \sum_{j=0}^r \binom{r}{j} [\{\vec^{-1}_{d^{r-j},p}\D^{\otimes r-j}\bbf(\bc)\}^\top \bu^{\otimes r-j}  \otimes \{\vec^{-1}_{d^{j},q}\D^{\otimes j}\bg(\bc)\}^\top \bu^{\otimes j} ] \\
&= \sum_{j=0}^r \binom{r}{j} [\{\vec^{-1}_{d^{r-j},p}\D^{\otimes r-j}\bbf(\bc)\}  \otimes \{\vec^{-1}_{d^{j},q}\D^{\otimes j}\bg(\bc)\}]^\top \bu^{\otimes r}.
\end{align*}
The desired formula then follows from Theorem \ref{thm:ident-vec}(ii).
\end{proof}

\begin{proof}[Proof of Theorem~\ref{thm:deriv-comp}]
(i) Reasoning as in \cite{Spi05}, it can be shown that $\D^{\otimes r}(g\circ \bbf)(\bc)$ depends only on the vectors $\D^{\otimes k}g\{\bbf(\bc)\}$ and $\D^{\otimes k}\bbf(\bc)$ for $k=1,\dots,r$, so that it suffices to show the theorem statement for any two functions $\tilde\bbf$ and $\tilde g$ that share these derivatives with $\bbf$ and $g$, respectively.

Without loss of generality we let $\bc=\mathbf 0$ and $\bbf(\bc)=\mathbf0$, and  we consider $\tilde\bbf(\bx)=\sum_{\ell=1}^r\{\bI_p\otimes(\bx^\top)^{\otimes\ell}\}\bm v_\ell$ and $\tilde g(\by)=\sum_{k=1}^r\bm w_k^\top\by^{\otimes k}$, for the given vectors $\bm v_\ell=\D^{\otimes \ell}\bbf(\mathbf 0)/\ell!\in\mathbb R^{pd^\ell}$ and $\bm w_k=\D^{\otimes k}g(\mathbf 0)/k!\in\mathbb R^{p^k}$. Then, by Corollary 3.1 in \citet{Hol85b}, the $k$th derivative of $\tilde\bbf$ at $\bx=\mathbf 0$ is $\D^{\otimes k}\tilde\bbf(\mathbf 0)=k!(\bI_p\otimes\S_{d,k})\bm v_k=\D^{\otimes k}\bbf(\mathbf 0)$, where the last equality is due to the fact that $\S_{d,k}\D^{\otimes k}f_i(\mathbf 0)=\D^{\otimes k}f_i(\mathbf 0)$ for $i=1,\dots, p$. Similarly, by Theorem 3.1 in \citet{Hol85b} we have $\D^{\otimes k}\tilde g(\mathbf 0)=k!\S_{p,k}\bm w_k=\D^{\otimes k}g(\mathbf 0)$ for $k=1,\dots,r$.

If we could express $\tilde{g}\circ\tilde{\bbf}$ as an $r$th order polynomial $(\tilde{g}\circ\tilde{\bbf})(\bx)=\sum_{\ell=1}^r(\bx^\top)^{\otimes \ell}\bb_\ell$ for vectors $\bb_\ell\in\mathbb R^{d^\ell}$, then we would have $\D^{\otimes r}(\tilde g\circ\tilde \bbf)(\mathbf 0)=r!\S_{d,r}\bb_r$.
So the proof would be complete if we could show that we can take the $r$th coefficient as
\begin{align}
\bb_r=&\frac1{r!}\sum_{k=1}^r\sum_{\bbm\in\mathcal J_{k,r}}\pi_\bbm k!(\bm w_k^\top\otimes\bI_{d^r})\bigotimes_{\ell=1}^r(\ell!\bm v_\ell)^{\otimes m_\ell}\nonumber\\&=\sum_{k=1}^r\sum_{\bbm\in\mathcal J_{k,r}}\frac{k!}{m_1!\cdots m_r!}(\bm w_k^\top\otimes\bI_{d^r})\bigotimes_{\ell=1}^r\bm v_\ell^{\otimes m_\ell} \in \mathbb{R}^{d^r}. \label{eq:ar}
\end{align}

In order to show Equation~\eqref{eq:ar}, denote by $\mathcal Q_{k,r}$ the set of partitions of $k$ into $r$ parts; that is, $\mathcal Q_{k,r}=\{\bbm\in\mathbb N_0^r\colon|\bbm|=k\}$. Define the symmetrizer matrix in dimension $p$ with respect to some $\bbm\in\mathcal Q_{k,r}$ as the only matrix $\S_{p,\bbm}\in\mathcal M_{p^k\times p^k}$ such that, for any $\bx_1,\dots,\bx_r\in\mathbb R^p$, the product $\frac{k!}{m_1!\cdots m_r!}\S_{p,\bbm}\bigotimes_{\ell=1}^r\bx_\ell^{\otimes m_r}$ equals the sum over all distinct terms of type $\bigotimes_{\ell=1}^r\by_\ell$ for which $m_1$ of the factors are equal to $\bx_1$, $m_2$ of the factors are equal to $\bx_2$, etc. Then, using the multinomial expansion in Equation (2.6) of \citet{Hol85b}, we obtain
\begin{align*}
(\tilde{g}\circ\tilde{\bbf})(\bx)&=\sum_{k=1}^r\bm w_k^\top\{\tilde{\bbf}(\bx)\}^{\otimes k}=\sum_{k=1}^r\bm w_k^\top\Big[\sum_{\ell=1}^r\{\bI_p\otimes(\bx^\top)^{\otimes\ell}\}\bm v_\ell\Big]^{\otimes k}\\
&=\sum_{k=1}^r\bm w_k^\top\sum_{\bbm\in\mathcal Q_{k,r}}\frac{k!}{m_1!\cdots m_r!}\S_{p,\bbm}\bigotimes_{\ell=1}^r[\{\bI_p\otimes(\bx^\top)^{\otimes\ell}\}\bm v_\ell]^{\otimes m_\ell}\\
&=\sum_{k=1}^r\sum_{\bbm\in\mathcal Q_{k,r}}\frac{k!}{m_1!\cdots m_r!}\bm w_k^\top\big\{\bI_{p^k}\otimes(\bx^\top)^{\otimes\sum_{\ell=1}^r\ell m_\ell}\big\}\bigotimes_{\ell=1}^r\bm v_\ell^{\otimes m_\ell}\\
&=\sum_{k=1}^r\sum_{\bbm\in\mathcal Q_{k,r}}\frac{k!}{m_1!\cdots m_r!}(\bx^\top)^{\otimes\sum_{\ell=1}^r\ell m_\ell}\big(\bm w_k^\top\otimes\bI_{d^{\sum_{\ell=1}^r\ell m_\ell}}\big)\bigotimes_{\ell=1}^r\bm v_\ell^{\otimes m_\ell},
\end{align*}
where we used that $\S_{p,\bbm}\bm w_k=\bm w_k$ by the definition of $\bm w_k$. This last equation shows that $\tilde{g}\circ\tilde{\bbf}$ is an $r$th order polynomial in $\bx$, where the vector that multiplies $(\bx^\top)^{\otimes r}$ is precisely $\bb_r$ in Equation~\eqref{eq:ar}, as desired.

(ii) If $\bg$ has components $(g_1,\dots,g_q)$, then $\bg\circ\bbf$ has components $(g_1\circ \bbf,\dots,g_q\circ\bbf)$, so the vector $\D^{\otimes r}(\bg\circ\bbf)(\bc)\in\mathbb R^{qd^r}$ is formed by stacking $\D^{\otimes r}(g_1\circ\bbf)(\bc),\dots,\D^{\otimes r}(g_q\circ\bbf)(\bc)$. By part (i),
$$\D^{\otimes r}(g_j\circ \bbf)(\bc) = \sum_{\bbm\in\mathcal J_{r}}\pi_{\bbm}\big[\D^{\otimes |\bbm|}g_j\{\bbf(\bc)\}^\top\otimes\S_{d,r}\big]\bigotimes_{\ell=1}^r\{\D^{\otimes \ell}\bbf(\bc)\}^{\otimes m_\ell}\in\mathbb R^{d^r}$$
for all $j=1,\dots,q$, so
$$\D^{\otimes r}(\bg\circ\bbf)(\bc)=\sum_{\bbm\in\mathcal J_{r}}\pi_{\bbm}\left( \begin{bbmatrix} \D^{\otimes |\bbm|}g_1\{\bbf(\bc)\}^\top \\ \hline \vdots \\ \hline \D^{\otimes |\bbm|}g_q\{\bbf(\bc)\}^\top \end{bbmatrix}\otimes\S_{d,r}\right)\bigotimes_{\ell=1}^r\{\D^{\otimes \ell}\bbf(\bc)\}^{\otimes m_\ell}.$$
The proof is completed by noting that
\begin{align*}
\begin{bbmatrix} \D^{\otimes |\bbm|}g_1\{\bbf(\bc)\}^\top \\ \hline \vdots \\ \hline \D^{\otimes |\bbm|}g_q\{\bbf(\bc)\}^\top \end{bbmatrix}&=
\left[\bgroup\def\arraystretch{0.8}\begin{array}{@{}c|c|c@{}} \D^{\otimes |\bbm|} g_1 \{\bbf(\bc)\} & \cdots &  \D^{\otimes |\bbm|} g_q \{\bbf(\bc)\}\end{array}\egroup\right]^\top\\&=\big[\vec_{p^{|\bbm|},q}^{-1} \D^{\otimes |\bbm|} \bg\{\bbf(\bc)\}\big]^\top.
\qedhere
\end{align*}
\end{proof}

\subsection{Matrix functions} \label{app:matrix}

The proofs in this section require a certain level of familiarity with established matrix algebra results, see \citet{Sch17,MN19} for an overview.

\begin{proof}[Proof of Lemma~\ref{lem:vec-matrix-prod}]
We begin with $\vec \{(\bA \bB)^r \bA\} = \{ \bA^\top \otimes (\vec^\top \bA)^{\otimes r-1} \otimes \bA\} (\vec \bB)^{\otimes r}$ for square matrices $\bA, \bB$.
This holds for $r=1$ since $\vec (\bA \bB \bA) = (\bA^\top \otimes \bA) \vec \bB$. The $(r+1)$th iteration is
\begin{align*}
\vec \{(\bA \bB)^{r+1} \bA\}
&= \vec \{\bA \bB (\bA \bB)^r \bA\} = (\bA^\top \otimes \bA) \vec \{ \bB (\bA\bB)^r\} \\
&= (\bA^\top \otimes \bA) \{\bB^\top \otimes (\vec^\top \bB)^{\otimes r-1} \otimes \bB\} (\vec \bA)^{\otimes r}
\end{align*}
from the induction hypothesis.
Continuing, using the known identities for expanding the vec of a Kronecker product as a Kronecker product of vectorized matrices \citep[Chapter~3.7]{MN19},
\begin{align*}
\vec &\{(\bA \bB)^{r+1} \bA\} \\
&= \{(\vec^\top \bA)^{\otimes r} \otimes \bA^\top \otimes \bA)\} \vec \{ \bB^\top \otimes (\vec^\top \bB)^{\otimes r-1} \otimes \bB\} \\
&= \{(\vec^\top \bA)^{\otimes r} \otimes \bA^\top \otimes \bA)\} (\bI_d \otimes \K_{d^{2r-1}, d} \otimes \bI_d) \{ \vec (\bB^\top) \otimes \vec [(\vec^\top \bB)^{\otimes r-1}  \otimes \bB]\} \\
&= \{(\vec^\top \bA)^{\otimes r} \otimes \bA^\top \otimes \bA)\} (\bI_d \otimes \K_{d^{2r-1}, d} \otimes \bI_d) \{ \vec (\bB^\top) \otimes (\vec \bB)^{\otimes r}\} \\
&= \{(\vec^\top \bA)^{\otimes r} \otimes \bA^\top \otimes \bA)\} (\bI_d \otimes \K_{d^{2r-1}, d} \otimes \bI_d) (\K_{d,d} \otimes \bI_{d^{2r}}) (\vec \bB)^{\otimes r+1} \\
&= \{(\vec^\top \bA)^{\otimes r} \otimes \bA^\top \otimes \bA)\} (\K_{d^{2r},d} \otimes \bI_d) (\vec \bB)^{\otimes r+1} \\
&= \{ \bA^\top \otimes (\vec^\top \bA)^{\otimes r} \otimes \bA)\}  (\vec \bB)^{\otimes r+1},
\end{align*}
where the last two equalities follow from $\K_{d^{2r},d} \otimes \bI_d = (\bI_d \otimes \K_{d^{2r-1},d} \otimes \bI_d) (\K_{d,d} \otimes \bI_{d^{2r}})$ from \citet[Theorem~8.29]{Sch17}, and $\K_{d,d^{2r}} \{(\vec \bA)^{\otimes r} \otimes \bA\}=\bA \otimes (\vec \bA)^{\otimes r}$  from \citet[Theorem~8.26(d)]{Sch17}. This finishes the induction argument.

If we replace the first and last $\bA$ in the product $\bA \bB(\bA \bB)^r \bA$ by $\bC$ and ${\mat D}$ so that product remains conformable, then we have $\vec \{\bC \bB (\bA \bB)^r {\mat D}\} = \{{\mat D}^\top \otimes (\vec^\top \bA)^{\otimes r} \otimes  \bC\}  (\vec \bB)^{\otimes r+1}$.  Exchanging the matrices $\bA$ and $\bC$ yields the result.
\end{proof}

\begin{proof}[Proof of Theorem~\ref{thm:deriv-monomial}]
(i) Since the Kronecker product is not commutative, the first differential of $(\vec\bX)^{\otimes k}$ can be decomposed as
\begin{align*}
\diff  \{(\vec\bX)^{\otimes k}\}&=\sum_{j=1}^k\big\{(\vec\bX)^{\otimes j-1}\otimes(\diff\vec\bX)\otimes(\vec\bX)^{\otimes k-j}\big\}\\
&=\Big\{\sum_{j=1}^k(\vec\bX)^{\otimes j-1}\otimes\bI_{b}\otimes(\vec\bX)^{\otimes k-j}\Big\}\diff\vec\bX\\
&=\Big(\sum_{j=1}^k\mat K_{b^{j},b^{k-j}}\Big)\big\{(\vec\bX)^{\otimes k-1}\otimes\bI_{b}\big\}\diff\vec\bX\\
&=\mat\Lambda_{b,k}\big\{(\vec\bX)^{\otimes k-1}\otimes\bI_{b}\big\}\diff\vec\bX,
\end{align*}
where $b=cd$ and $\mat\Lambda_{b,k}=\sum_{j=1}^k\mat K_{b^{j},b^{k-j}}\in\mathcal M_{b^{k}\times b^{k}}$.
If we repeatedly iterate this first differential, then for $1\leq r\leq k$, the $r$th differential is
$$\diff^r\{(\vec\bX)^{\otimes k}\}=\mat\Gamma_{b,k,r}\big\{(\vec\bX)^{\otimes k-r}\otimes\bI_{b^{r}}\big\}(\diff\vec\bX)^{\otimes r},$$
where $\mat\Gamma_{b,k,r}=\prod_{j=0}^{r-1}(\mat\Lambda_{b,k-j}\otimes\bI_{b^{j}}) \in\mathcal M_{b^{k}\times b^{k}}$ is the product of matrices given in the statement of Theorem \ref{thm:deriv-monomial}. Applying Theorem~\ref{thm:ident-vec}, the $r$th derivative is
\begin{align*}
\D^{\otimes r} \{(\vec\bX)^{\otimes k}\}&=(\bI_{b^{k}}\otimes\S_{b,r})\vec\big[\big\{(\vec^\top\bX)^{\otimes k-r}\otimes\bI_{b^{r}}\big\}\mat\Gamma_{b,k,r}^\top\big]\\
&=(\bI_{b^{k}}\otimes\S_{b,r})(\mat\Gamma_{b,k,r}\otimes\bI_{b^{r}})\big\{(\vec\bX)^{\otimes k-r}\otimes\vec\bI_{b^{r}}\big\}\\
&=(\mat\Gamma_{b,k,r}\otimes\S_{b,r})\big\{(\vec\bX)^{\otimes k-r}\otimes\vec\bI_{b^{r}}\big\}.
\end{align*}

(ii) This follows from (i) with $c=1$, so that $b=d$.

(iii) For the matrix product $\bX^k$, we first appeal to Lemma~\ref{lem:vec-matrix-prod} to obtain $\vec (\bX^k) = \{\bI_d \otimes (\vec^\top \bI_d)^{\otimes k-1} \otimes \bI_d\} (\vec \bX)^{\otimes k}$.
Applying Theorem~\ref{thm:deriv-const}, then $\D^{\otimes r} \vec (\bX^k) = \{\bI_d \otimes (\vec^\top \bI_d)^{\otimes k-1} \otimes \bI_{d^{2r+1}}\} \D^{\otimes r} \{(\vec \bX)^{\otimes k}\}$, and substituting $\D^{\otimes r} \{(\vec \bX)^{\otimes k}\}$ from (i) above, it follows that
\begin{align*}
\D^{\otimes r} \vec (\bX^k)
&=  \{\bI_d \otimes (\vec^\top \bI_d)^{\otimes k-1} \otimes \bI_{d^{2r+1}}\} (\mat\Gamma_{d^2,k,r}\otimes\S_{d^2,r})\big\{(\vec\bX)^{\otimes k-r}\otimes\vec\bI_{d^{2r}}\big\}.
\end{align*}

(iv)
The derivatives for the trace are straightforward to compute from the identity $\tr (\bX^k) = (\vec^\top \bI_d) \vec (\bX^k)$. Via Theorem~\ref{thm:deriv-const}, and using part (iii),
\begin{align*}
\D^{\otimes r} \tr (\bX^k)
&= \D^{\otimes r} \{ (\vec^\top \bI_d) (\vec \bX^k)\} = \{(\vec^\top \bI_d) \otimes \bI_{d^{2r}}\} \D^{\otimes r} \vec (\bX^k)\\
&=[\{(\vec^\top \bI_d)\mat\Upsilon_{d,k,r}\}\otimes\S_{d^2,r}]\big\{(\vec\bX)^{\otimes k-r}\otimes\vec\bI_{d^{2r}}\big\}.
\end{align*}

(v) Finally, we prove that for $r=1$ the above formula coincides with the simpler expression that can be obtained from the differential given in \citet[p. 186]{MN19}. Let ${\bv_k}=\D \tr (\bX^k)$ be as given in Theorem~\ref{thm:deriv-monomial}(iv) with $r=1$. The goal is to show that $\bv_k= k \vec \{(\bX^\top)^{k-1}\}$. This holds for $k=1$, since $\mat\Upsilon_{d,1,1}=\bI_{d^2}$ so that $\bv_1=\{(\vec^\top \bI_d)\otimes\bI_{d^2}\}\vec\bI_{d^2}=\vec\bI_d$. Inductively, assuming that the result is true for some $k$, the next iteration is provided by
$$\bv_{k+1} = \D \tr (\bX^k \bX) = [\vec^\top \{(\bX^\top)^k\} \otimes \bI_{d^2}] \D (\vec \bX) + \{\vec^\top (\bX^\top) \otimes \bI_{d^2}\} \D (\vec \bX^k)$$
from Theorem~\ref{thm:deriv-const}.
Applying Theorem~\ref{thm:deriv-monomial}(iii) with $r=1$ to $\vec\bX$, then we can assert that the first term of $\bv_{k+1}$ is $[\vec^\top \{(\bX^\top)^k\} \otimes \bI_{d^2}] \vec \bI_{d^2} = \vec \{(\bX^\top)^k\}$. Applying Theorem~\ref{thm:deriv-monomial}(iii) to $\vec\bX^k$, then the second term is
\begin{align*}
[\vec^\top (\bX^\top) \otimes \bI_{d^2}] (\mat\Upsilon_{d,k,1}&\otimes\bI_{d^2})\big\{(\vec\bX)^{\otimes k-1}\otimes\vec\bI_{d^2}\big\}\\
 &= (\bX \otimes \bI_d)  [\{(\vec^\top \bI_{d}) \mat\Upsilon_{d,k,1}\}\otimes\bI_{d^2}]\big\{(\vec\bX)^{\otimes k-1}\otimes\vec\bI_{d^2}\big\}\\
 &= (\bX \otimes \bI_d) \bv_k = k (\bX \otimes \bI_d) \vec \{(\bX^\top)^{k-1}\},
\end{align*}
where the last two equalities follow from Theorem~\ref{thm:deriv-monomial}(iv) and the induction hypothesis. Thus the second term of $\bv_{k+1}$ is $k\vec \{(\bX^\top)^{k}\}$. Therefore $\bv_{k+1} = (k+1) \vec \{(\bX^\top)^{k}\}$ which completes the induction.
\end{proof}

\begin{proof}[Proof of Theorem~\ref{thm:matrix-inv}]
We begin with the matrix differential
\begin{equation}\label{eq:matrix-inv}
\diff^r (\bX^{-1}) =  (-1)^r r! (\bX^{-1} \diff \bX)^r \bX^{-1}
\end{equation}
which was stated in \citet[p.~169]{MN19}.
%
%
Applying the vectorization operator to both sides of Equation~\eqref{eq:matrix-inv}, and using Lemma~\ref{lem:vec-matrix-prod} to expand the right hand side, results in
$$\diff^r \vec (\bX^{-1}) =  (-1)^r r!\{(\bX^{-1})^\top\otimes(\vec^\top\bX^{-1})^{\otimes r-1}\otimes\bX^{-1}\}(\diff \vec\bX)^{\otimes r}.$$
Applying Theorem~\ref{thm:ident-vec}, we have
\begin{align*}
\D^{\otimes r} (\vec \bX^{-1})
&= (-1)^r r! (\bI_{d^2} \otimes \S_{d^2,r}) \vec \{\bX^{-1} \otimes (\vec \bX^{-1})^{\otimes r-1} \otimes (\bX^{-1})^\top\} \\
&= (-1)^r r! (\bI_{d^2} \otimes \S_{d^2,r}) \K_{d^2, d^{2r}} \vec \{(\bX^{-1})^\top \otimes (\vec^\top \bX^{-1})^{\otimes r-1} \otimes \bX^{-1}\} \\
&= (-1)^r r! (\bI_{d^2} \otimes \S_{d^2,r}) (\bI_d \otimes \K_{d,d^{2r}})(\K_{d,d^{2r}} \otimes \bI_d) (\K_{d^{2r},d} \otimes \bI_d) (\vec \bX^{-1})^{\otimes r+1} \\
&= (-1)^r r! (\bI_{d^2} \otimes \S_{d^2,r}) (\bI_d \otimes \K_{d,d^{2r}}) (\vec \bX^{-1})^{\otimes r+1}
\end{align*}
using a similar calculation to Lemma~\ref{lem:vec-matrix-prod}, and that $\K_{m,n} \K_{n,m} = \bI_{mn}$ \cite[Theorem~8.24]{Sch17}.
\end{proof}

\begin{proof}[Proof of Theorem~\ref{thm:matrix-det}]
	(i) \citet[][p. 180]{MN19} provided the differential
	$$\diff^r\log|\bX|=(-1)^{r-1}(r-1)!\tr\{(\bX^{-1}\diff\bX)^r\}.$$
	And, using Lemma \ref{lem:vec-matrix-prod},
	\begin{align*}
		\tr\{(\bX^{-1}\diff\bX)^r\}&=(\vec^\top\bI_d)\vec\{(\bX^{-1}\diff\bX)^r\}\\
		&=(\vec^\top\bI_d)\{\bI_d\otimes(\vec^\top\bX^{-1})^{\otimes r-1}\otimes\bX^{-1}\}(\diff\vec\bX)^{\otimes r}\\
		&=\vec^\top\big\{(\vec\bX^{-1})^{\otimes r-1}\otimes(\bX^{-1})^\top\big\}(\diff\vec\bX)^{\otimes r}.
	\end{align*}
	Therefore, applying Theorem \ref{thm:ident},
	\begin{align*}
		\D^{\otimes r}\log|\bX|&=(-1)^{r-1}(r-1)!\S_{d^2,r}\vec\big\{(\vec\bX^{-1})^{\otimes r-1}\otimes(\bX^{-1})^\top\big\}\\
		&=(-1)^{r-1}(r-1)!\S_{d^2,r}(\mat K_{d,d^{2r-2}}\otimes\bI_d)\big[(\vec\bX^{-1})^{\otimes r-1}\otimes\vec\{(\bX^{-1})^\top\}\big]\\
		&=(-1)^{r-1}(r-1)!\S_{d^2,r}(\mat K_{d,d^{2r-2}}\otimes\bI_d)(\bI_{d^{2r-2}}\otimes\mat K_{dd})(\vec\bX^{-1})^{\otimes r}\\
		&=(-1)^{r-1}(r-1)!\S_{d^2,r}\mat K_{d,d^{2r-1}}(\vec\bX^{-1})^{\otimes r},
	\end{align*}
	where the last equality follows from a transposed form of Theorem 8.29 in \cite{Sch17}.

(ii) We can write $|\bX|=(g\circ f)(\bX)$, where $f(\bX)=\log|\bX|$ and $g(t)=e^t$, so that $\D^{\otimes m} g\{f(\bX)\}=|\bX|$ for all $m\in\mathbb N$. Then, combining Theorem \ref{thm:matrix-det}(i) with Fa\`{a} di Bruno's formula from Theorem \ref{thm:deriv-comp}, we obtain
\begin{align*}
\D^{\otimes r}|\bX|&=\S_{d^2,r}|\bX|\sum_{\bbm\in\mathcal J_{r}}\pi_{\bbm}\bigotimes_{\ell=1}^r\{\D^{\otimes \ell} \log |\bX|\}^{\otimes m_\ell}\\
&=\S_{d^2,r}|\bX|\sum_{\bbm\in\mathcal J_{r}}\pi_{\bbm}\bigotimes_{\ell=1}^r\big\{(-1)^{\ell-1}(\ell-1)!\S_{d^2,\ell}\mat K_{d,d^{2\ell-1}}(\vec\bX^{-1})^{\otimes \ell}\big\}^{\otimes m_\ell}\\
&=\S_{d^2,r}|\bX|\sum_{\bbm\in\mathcal J_{r}}(-1)^{r-|\bbm|}\frac{r!}{\prod_{\ell=1}^r(m_\ell!\ell^{m_\ell})}\bigotimes_{\ell=1}^r\big(\S_{d^2,\ell}\mat K_{d,d^{2\ell-1}}\big)^{\otimes m_\ell}(\vec\bX^{-1})^{\otimes r}.
\end{align*}
Finally, we assert that $\S_{d^2,r}\bigotimes_{\ell=1}^r\S_{d^2,\ell}^{\otimes m_\ell}=\S_{d^2,r}$: the reason for this is that $\S_{d^2,r}$ makes the $r$th order Kronecker product commutative, and for any arbitrary $\bv_1,\dots,\bv_r\in\mathbb R^{d^2}$ the product $\bigotimes_{\ell=1}^r\S_{d^2,\ell}^{\otimes m_\ell}\bigotimes_{\ell=1}^r\bv_\ell$ is the mean of a number of terms which are all of the form $\bigotimes_{\ell=1}^r\bv_{\sigma(\ell)}$ for some $\sigma\in\mathcal P_r$, so that eventually $\S_{d^2,r}\bigotimes_{\ell=1}^r\S_{d^2,\ell}^{\otimes m_\ell}\bigotimes_{\ell=1}^r\bv_\ell=\S_{d^2,r}\bigotimes_{\ell=1}^r\bv_\ell$. Since the $\bv_1,\dots,\bv_r$ are arbitrary, this implies that $\S_{d^2,r}\bigotimes_{\ell=1}^r\S_{d^2,\ell}^{\otimes m_\ell}=\S_{d^2,r}$, as desired.
\end{proof}

\subsection{Vectorized moments} \label{app:moments}

\begin{lemma}\label{lem:mom}
Suppose that $\mathbb E(|X_{i_1}\cdots X_{i_r}|)<\infty$ for any choice of $i_1,\dots,i_r\in\{1,\dots,d\}$. Then $\D^{\otimes r} M_\bX (\bt) = \E \{ \exp(\bt^\top\bX) \bX^{\otimes r} \}$.
\end{lemma}

\begin{proof}
The condition on the absolute mixed moments implies that we can change the order of differentiation and expectation \citep[see][Section 26]{B12}. So it suffices to find the derivative of $\alpha (\bt) = \exp (\bt^\top \bx) = (g \circ f) (\bt)$ where $g(t) = \exp(t)$ and $f(\bt) = \bt^\top \bx$. We have $\D^{\otimes r} g(t) = g(t)$ and $ \D^{\otimes r} g\{f (\bt)\} = \alpha(\bt)$  for all $r$. Besides, $\D f(\bt) = \bx$ and $\D^{\otimes  r} f(\bt) = \bm 0$ for all $r\geq 2$. Reasoning as for the Hermite polynomial calculation, the only required multi-index, which is a solution of $1\cdot m_1 + \dots + r \cdot m_r = r$, is $m_1=r, m_2 = \dots = m_r =0$.  Since the $\pi$ coefficient is $\pi_\bbm = r!/r! = 1$, then Theorem~\ref{thm:deriv-comp}(i) implies that
$\D^{\otimes r} \alpha(\bt)
= \alpha(\bt) \pi_{\bbm} \S_{d,r} \{\D f(\bt)\}^{\otimes m_1} = \exp( \bt^\top \bx) \S_{d,r} \bx^{\otimes r} =  \exp( \bt^\top \bx) \bx^{\otimes r}$,
as required.
\end{proof}

\end{document}